\documentclass[10pt,reqno,francais]{smfart}
\usepackage{smfthm}
\usepackage{bull}

\usepackage{amsmath,amsxtra,amssymb,amscd,amsfonts} 

\usepackage{latexsym,array,multirow,makecell}

\usepackage[vietnam,francais]{babel}
\usepackage[utf8]{inputenc}
\usepackage[T1]{fontenc}
\usepackage{lmodern}

\usepackage{multirow}
\usepackage{hhline}

\usepackage{graphicx}
\usepackage{shuffle}
\usepackage[mathcal]{euscript}
\usepackage{faktor}
\usepackage{mathrsfs}
\usepackage{bm}
\usepackage{longtable}
\usepackage[multiple]{footmisc}
\usepackage[normalem]{ulem} 
\usepackage{etex} 

\usepackage{indentfirst}
\usepackage{pb-diagram}
\usepackage[all,cmtip]{xy}

\usepackage[usenames,dvipsnames]{color}
\usepackage[table]{xcolor}
\usepackage{colortbl}
\newcolumntype{a}{>{\columncolor{Gray}}c}
\newcolumntype{b}{>{\columncolor{white}}c}

\usepackage{pgf,tikz}
\usetikzlibrary{arrows,decorations.pathreplacing,matrix,positioning,automata,spy,shapes.geometric,backgrounds,calc,trees,shapes,fadings,decorations.pathmorphing}
\usepackage{fancybox}
\usepackage{wrapfig}
\usetikzlibrary[patterns]
\definecolor{zzttqq}{rgb}{0.6,0.2,0.0}
\definecolor{qqqqff}{rgb}{0.0,0.0,1.0}
\definecolor{cqcqcq}{rgb}{0.7529411764705882,0.7529411764705882,0.7529411764705882}
\definecolor{Gray}{gray}{0.85}
\definecolor{LightCyan}{rgb}{0.88,1,1}

\allowdisplaybreaks

\def\kc{\Bbbk}

\def\F{\mathcal F}
\def\f2{\mathbb F_2}
\def\fp{\mathbb F_p}
\def\a2{\mathcal A_2}

\def\v{\mathcal V}

\def\h{\textup{H}}
\def\n{\mathbb N}

\def\z{\mathbb Z}

\def\nil{\mathcal{N}il}

\def\u{\mathcal U}

\def\p{\mathcal P}

\newcommand{\cke}[1]{\mathrm{Coker}\left(#1\right)}

\newcommand{\ex}[4]{\mathrm{Ext}_{#1}^{#2}\left(#3,#4\right)}

\newcommand{\ho}[3]{\mathrm{Hom}_{#1}\left(#2,#3\right)}

\newcommand{\ke}[1]{\mathrm{Ker}\left(#1\right)}

\newcommand{\dsum}[3]{\underset{#1}{\overset{#2}{\bigoplus}}#3}

\newcommand{\rd}[2]{\mathrm{R}^{#1}\left(#2\right)}

\newcommand{\ro}[3]{\rho_{#1,#2}\left(#3\right)}

\newcommand{\pa}[2]{\partial^{#1}\left(#2\right)}
\newcommand{\morp}[2]{#1\left(#2\right)}
\newcommand{\de}[1]{\mathrm{d}\left(#1\right)}

\newtheorem{notat}{Notation}
\newtheorem{conex}{Contre-exemple}

\newcounter{nmdthmcnt}
\newenvironment{namedthm}[2][]{\addtocounter{nmdthmcnt}{1}
    \newtheorem*{nmdthm\roman{nmdthmcnt}}{Théorème #2}%
    \begin{nmdthm\roman{nmdthmcnt}}[#1]}{\end{nmdthm\roman{nmdthmcnt}}}

\newenvironment{namedcor}[2][]{\addtocounter{nmdthmcnt}{1}
    \newtheorem*{nmdthm\roman{nmdthmcnt}}{Corollaire #2}%
    \begin{nmdthm\roman{nmdthmcnt}}[#1]}{\end{nmdthm\roman{nmdthmcnt}}}
    
\newenvironment{namedpro}[2][]{\addtocounter{nmdthmcnt}{1}
    \newtheorem*{nmdthm\roman{nmdthmcnt}}{Proposition #2}%
    \begin{nmdthm\roman{nmdthmcnt}}[#1]}{\end{nmdthm\roman{nmdthmcnt}}}

\author[NGUYEN The Cuong]{\selectlanguage{vietnam} Thế Cường NGUYỄN}
\address{Laboratoire Analyse, Géométrie et Applications - UMR7539 du CNRS\\99 Avenue Jean-Baptiste Clément\\93430 Villetaneuse\\France}
\email{tdntcuong@gmail.com ou nguyentc@math.univ-paris13.fr}
\urladdr{http://math.univ-paris13.fr/~nguyentc/Pageweb.html}    
\title[Sur la torsion de Frobenius de la catégorie $\u$]{Sur la torsion de Frobenius de la catégorie des modules instables} 
\alttitle{On the Frobenius twist in the category of unstable modules}

\begin{document}
\frontmatter
\begin{abstract}
Un des phénomènes marquants dans la catégorie $\p_{d}$ des foncteurs polynomiaux stricts est l'injectivité des morphismes induits par la torsion de Frobenius entre groupes d'extensions des foncteurs. Dans \cite{Cuo14}, l'auteur démontre que le foncteur de Hai, allant de la catégorie $\p_{d}$ vers la catégorie des modules instable $\u$, est pleinement fidèle. Cela fait de la catégorie $\p_{d}$ une sous-catégorie pleine de la catégorie $\u$. La torsion de Frobenius s'étend à toute la catégorie $\u$,
mais n'y est pas aussi bien comprise. Cet article étudie la torsion de Frobenius, soit dans ce cas le foncteur double $\Phi$, et ses effets sur les groupes d'extension des modules instables. On donne des calculs explicites de nombreux groupes d'extensions des modules instables, 
et permet de confirmer, dans de nombreux cas,
l'injectivité des morphismes entre des groupes d'extensions induits par la torsion de Frobenius dans la catégorie $\u$.
Ces résultats sont obtenus en étudiant la résolution injective minimale du module instable libre $F(1)$.  

\end{abstract}
\begin{altabstract}
In the category $\p_{d}$ of strict polynomial functors, the morphisms between extension groups induced by the Frobenius twist are injective. In \cite{Cuo14}, the category $\p_{d}$ is proved to be a full sub-category of the category $\u$ of unstable modules \textit{via} Hai's functor. The Frobenius twist is extended to the category $\u$ but remains mysterious there. This article aims to study the Frobenius twist $\Phi$ of the category $\u$ and its effects on the extension groups of unstable modules. We compute explicitly several extension groups and show that in these cases, the morphisms induced by the Frobenius twist are injective. These results are obtained by constructing the minimal injective resolution of the free unstable module $F(1)$.

\end{altabstract}
\altkeywords{Frobenius twist, module nilpotent, résolution injective, Steenrod algebra, strict polynomial functors, unstable modules}
\translator{The Cuong NGUYEN}
\thanks{L'auteur est partiellement soutenu par le programme ARCUS Vietnam MAE, Région IDF et par LIAFV - CNRS - Formath Vietnam}
\maketitle
\tableofcontents 
\mainmatter

\section{Introduction}
Soit $\kc$ un corps de caractéristique $p>0,$ et soit $V$ un $\kc-$espace vectoriel. On note $F$ pour l'homomorphisme de Frobenius $x\mapsto x^{p}.$ La torsion de Frobenius de $V,$ notée $V^{(1)},$ est le $\kc-$espace vectoriel, qui comme groupe abélien, s'identifie à $V$ mais, dont la multiplication par les scalaires est donnée par $$\lambda\cdot v=\lambda^{p}v.$$
\'{E}tant donné un $GL_{n}(\kc)-$module $M$, sa torsion de Frobenius $M^{(1)}$ est le $GL_{n}(\kc)-$module induit par l'homomorphisme de Frobenius 
$$GL_{n}(\kc)\xrightarrow{F}GL_{n}(\kc).$$

D'après \cite{CPS83,Jan87}\footnote{Je tiens à remercier Professeur Wilberd van der Kallen pour m'avoir indiqué la bonne réference pour ce résultat.}, la torsion de Frobenius induit des monomorphismes entre des groupes d'extensions des $GL_{n}(\kc)-$modules:
$$\ex{GL_{n}(\kc)}{*}{M}{N}\hookrightarrow\ex{GL_{n}(\kc)}{*}{M^{(1)}}{N^{(1)}}.$$
Selon Friedlander et Suslin \cite[corollaire 3.13]{FS97}, les groupes d'extensions des $GL_{n}(\kc)-$modules se calculent comme ceux des foncteurs polynomiaux stricts. On a aussi une torsion de Frobenius $\p_{d}\xrightarrow{(-)^{(1)}}\p_{pd}$ dans la catégorie $\bigoplus_{d\geq 0}\p_{d}$ des foncteurs polynomiaux stricts, et celle-ci induit donc des monomorphismes sur les groupes d'extensions. Ce résultat est crucial dans \cite{FFSS99} pour démontrer que les groupes d'extensions des foncteurs polynomiaux stricts se stabilisent par rapport aux itérés de la torsion de Frobenius.

Soit $\F$ la catégorie des foncteurs de la catégorie des $\fp-$espaces vectoriels de dimension finie vers la catégorie des $\fp-$espaces vectoriels. Il n'y a pas de torsion de Frobenius non-triviale dans $\F$. Le foncteur oubli $\bigoplus_{d\geq 0}\p_{d}\to \F$ se factorise à travers de la catégorie $\u$ des modules instables \cite{Hai},\textit{ via }un foncteur exact:
$$\bar{m}:\bigoplus_{d\geq 0}\p_{d}\to \u$$
La restriction $\bar{m}_{d}$ de ce foncteur sur $\p_{d}$ est pleinement fidèle \cite{Cuo14}. Il y a sur $\u$ un avatar de la torsion de Frobenius: le foncteur double $\Phi.$ Nguyen D. H. Hai \cite{Hai} montre que:
$$\Phi \bar{m}_{d}(F)\cong \bar{m}_{d}\left(F^{(1)}\right).$$ 
Il est donc naturel de se poser la question de l'injectivité de la torsion de Frobenius sur les groupes d'extensions des modules instables. Celle-ci ne peut avoir lieu en toute généralité, ne serait ce qu'à cause des modules instables nilpotents. Des contre-exemples seront donnés avec des modules instables $\nil-$fermés aussi. 

Cependant des calculs effectués, à l'aide du module instable $F(1)\cong \bar{m}_{1}(I),I$ désignant le foncteur d'inclusion, sur le système
$$\cdots\to\ex{\u}{*}{\Phi^{n}F(1)}{\Phi^{n}F(1)}\to \ex{\u}{*}{\Phi^{n+1}F(1)}{\Phi^{n+1}F(1)}\to\cdots$$
montrent qu'il subsiste des propriétés intéressantes:
\begin{namedthm}{\ref{princh44}}
	Pour $i\leq 49$ ou $i=2^{n}-2^{5}+t$ avec $0\leq t\leq 2^{5}+2$ et $n>5$, il y a des monomorphismes
	$$\ex{\u}{i}{\Phi^{r} F(1)}{\Phi^{r} F(1)}\hookrightarrow\ex{\u}{i}{\Phi^{r+1} F(1)}{\Phi^{r+1} F(1)}$$
	pour tout $r$.
\end{namedthm}
Pour effectuer ces calculs, on a besoin d'étudier la résolution injective (minimale de préférence) de $F(1).$ On la désigne par $\left(I^{r},\partial^{r} \right)^{r\geq 0}.$ Nous ne la connaissons pas, mais nous pouvons dire beaucoup de choses à ce propos. 

Rappelons les objets injectifs de la catégorie $\u$. On désigne par $J(n)$ l'enveloppe injective de la cohomologie réduite $\tilde{\h}^{*}\left(S^{n};\f2\right)$ du sphère $S^{n}$. Les $J(n)$ sont injectifs, caractérisés par $\ho{\u}{M}{J(n)}\cong \left( M^{i}\right)^{*}.$ Soit $V$ un $2-$groupe abélien élémentaire, la cohomologie de l'espace classifiant $BV,$ que l'on note  $\h^{*}V$ (on note $\tilde{\h}^{*}V$ pour la cohomologie réduite), est injectif dans la catégorie $\u$ d'après un résultat de Miller dans sa preuve de la conjecture de Sullivan \cite{Mil84}. Les travaux de Franjou, Lannes et Schwartz, \cite[théorème 3.1]{LS89},\cite[théorème 7.3]{FLS94} permettent de démontrer:
\begin{namedcor}{\ref{lemmeimp}}
Pour tout $r\in \mathbb{N}$, le module $I^{r}$ se décompose en somme directe $R^{r}\oplus N^{r}$ où $R^{r}$ est un facteur direct d'une certaine somme directe $\bigoplus_{\alpha}\h^{*}\left( BV_{\alpha};\f2\right)$ et $N^{r}$ est une somme directe finie de modules de Brown-Gitler.
\end{namedcor}
Puisque le groupe $\ho{\u}{J(n)}{\h^{*}V}$ est trivial, la suite  $\left(N^{r},\partial^{r}|_{N^{r}} \right)^{r\geq 0}$ est un sous-complexe de la résolution injective minimale de $F(1)$. Le calcul de la cohomologie de ce sous-complexe dépend de la cohomologie de MacLane des corps finis \cite{FLS94}.
\begin{namedcor}{\ref{calcul1}}
Soit $r=2^{k}(2l+1)$, alors:
$$\h^{r}\left( N^{\bullet},\partial^{\bullet}|_{N^{\bullet}}\right)\cong \frac{F(1)}{\Phi^{k} F(1)}.$$
\end{namedcor}
Soient $k>l\geq 2$. On observe que si $t$ est un entier tel que $t<2^{l}$ alors 
$$\h^{t+2^{l}}\left( N^{\bullet},\partial^{\bullet}|_{N^{\bullet}}\right)\cong \h^{t+2^{k}-2^{l}}\left( N^{\bullet},\partial^{\bullet}|_{N^{\bullet}}\right).$$
Cette périodicité particulière de la cohomologie $\h^{*}\left( N^{\bullet},\partial^{\bullet}|_{N^{\bullet}}\right)$ se relève au complexe $\left( N^{\bullet},\partial^{\bullet}|_{N^{\bullet}}\right)^{r\geq 0}$.
\begin{namedthm}{\ref{princh42}
}
Sous un même hypothèse sur $k,l$ et $t$, il y a des isomorphismes:
$$N^{t+2^{l}}\cong N^{t+2^{k}-2^{l}}.$$
\end{namedthm}

Afin d'énoncer le résultat sur le sous-complexe $\left( N^{\bullet},\partial^{\bullet}|_{N^{\bullet}}\right)$ de la résolution injective minimale de $F(1)$, posons:
\begin{align*}
J(n_1,\ldots,n_k) :=\dsum{i=1}{k}{J(n_k)},
\end{align*}
et
\begin{align*}
A_{2^{n}+3}:=\left(\bigoplus_{i=1}^{n-3}J\left(2^{n-1}-2^{i}\right)\right)\bigoplus \left(\bigoplus_{\substack{1\leq i\leq n-2\\0\leq j\leq i-2}}^{}J\left(2^{n-1}-2^{i}-2^{j}\right)\right).
\end{align*}
\begin{namedthm}{\ref{princh43}}
Pour $n\geq 6$, on a:
	\begin{center}
		\fontsize{9pt}{11pt}\selectfont
		\begin{longtable}{|c||*{7}{c|}}
			\hline
			\rowcolor{Gray}\makebox[0.7em]{$k$}& \makebox[0.7em]{$2^{n}-32$}& \makebox[0.7em]{$2^{n}-31$}& \makebox[0.7em]{$2^{n}-30$}& \makebox[0.7em]{$2^{n}-29$}&\makebox[2.3em]{$2^{n}-28$}\\\hline
			$N^{k}$ & $0$& $J(16)$& $J(15,14,12)$ &$J(14,12,11,8)$ & $J(13,10,5)$\\\hline\hline
			\rowcolor{Gray} \makebox[0.7em]{$k$}&\makebox[2.3em]{$2^{n}-27$}& \makebox[0.7em]{$2^{n}-26$}& \makebox[0.7em]{$2^{n}-25$}& \makebox[0.7em]{$2^{n}-24$}& \makebox[0.7em]{$2^{n}-23$} \\\hline
			$N^{k}$& $J(12,4,3)$ & $J(11,6,3)$& $J(10,2)$& $J(9)$& $J(8)$\\\hline
			\rowcolor{Gray}$k$ & \makebox[0.7em]{$2^{n}-22$}& \makebox[0.7em]{$2^{n}-21$}& $2^{n}-20$& $2^{n}-19$& $2^{n}-18$\\\hline
			$N^{k}$ & $J(7,6)$ & $J(6,4)$& $J(5)$& $J(4)$& $J(3)$\\\hline\hline
			\rowcolor{Gray}$k$& $2^{n}-17$& $2^{n}-16$& $2^{n}-15$& $2^{n}-14$& $2^{n}-13$\\ \hline
			$N^{k}$& $J(2)$& $0$&$J(8)$ & $J(7,6)$ & $J(6,4)$\\\hline
			\rowcolor{Gray}$k$& $2^{n}-12$& $2^{n}-11$& $2^{n}-10$& $2^{n}-9$& $2^{n}-8$\\ \hline
			$N^{k}$& $J(5)$ & $J(4)$ & $J(3)$ & $J(2)$ & $0$\\\hline
			\rowcolor{Gray}$k$& $2^{n}-7$& $2^{n}-6$& $2^{n}-5$& $2^{n}-4$& $2^{n}-3$\\ \hline
						$N^{k}$& $J(4)$& $J(3)$&$J(2)$ & $0$ & $J(2)$\\\hline
						\rowcolor{Gray}$k$& $2^{n}-2$& $2^{n}-1$& $2^{n}$& \multicolumn{2}{c|}{$2^{n}+1$}\\ \hline
						$N^{k}$& $0$ & $J(1)$ & $0$ & \multicolumn{2}{c|}{$J(2^{n-1})$}\\\hline
	\rowcolor{Gray}$k$& \multicolumn{3}{c|}{$2^{n}+2$}  & \multicolumn{2}{c|}{$2^{n}+3$}   \\ \hline
							$N^{k}$& \multicolumn{3}{c|}{$J(2^{n-1}-1,2^{n-1}-2,\ldots,2^{n-1}-2^{n-3})$} &  \multicolumn{2}{c|}{$J(2^{n-2})\oplus A_{2^{n}+3}$} \\\hline
		\end{longtable}
	\end{center}
\end{namedthm}

On observe que dans la zone où on peut expliciter $N^{k}$, les modules de Brown-Gitler du type $J(2^{n})$ sont répartis de la manière suivante: chaque module $N^{2l+1}$ contient un seul facteur de ce type et les modules $N^{2l}$ n'en contiennent aucun. Cette particularité implique:
\begin{namedpro}{\ref{calculessentiel}}
Soit $l$ un entier. Si il existe un entier $n_{l}$ tel que  $$N^{2l+1}=J(2^{n_{l}})\oplus
\bigoplus_{\alpha}J\left( 2^{m_{\alpha}}(2t_{\alpha}+1)\right)$$ et que $N^{2l}$ ne contient
aucun facteur direct de type $J(2^{n})$ alors on a:
$$\
\ex{\u}{2l}{\Phi^{k}F(1)}{F(1)}\cong\left\{\begin{array}{cl}
0 &\text{ si } k\leq n_{l},\\
\f2&\text{ si } k> n_{l}.
\end{array}\right.
$$
De plus, les morphismes $$\ex{\u}{2l}{\Phi^{k}F(1)}{F(1)}\to \ex{\u}{2l}{\Phi^{k+1}F(1)}{F(1)}$$
sont injectifs.
\end{namedpro}
\subsection*{Plan de l'article} Dans la section 2 on rappelle des généralités sur les modules instables injectifs et considère la partie réduite de la résolution injective minimale du module $F(1)$. On étudie la partie nilpotente de cette résolution dans la section 3, on y montre que chaque terme de cette partie est somme directe finie de modules de Brown-Gitler. On découvre un phénomène de périodicité de la partie nilpotente dans la section 4. La section 5 est consacrée pour déterminer les extensions de la partie nilpotente par la partie réduite. Les cas particuliers de la partie nilpotente se calculent dans la section 6. La dernière section a pour but de calculer les groupes d'extensions des modules instables et étudier les effets de la torsion de Frobenius sur eux. Une conjecture sur ces effets est donnée à la fin de l'article.
\subsection*{Remerciements} Ce travail fait partie de ma thèse de doctorat effectuée à l'Université Paris 13 sous la direction de Lionel Schwartz. Je tiens à remercier Professeur Lionel Schwartz pour sa générosité, son soutien constant, ses précieux conseils, sa patience et ses exigences qui m'ont toujours apporté autant humainement que scientifiquement. Mes remerciements vont également à Vincent Franjou. Ses remarques m'ont permis d'envisager ce travail sous un autre angle. J'aimerais profiter de cette occasion pour remercier les membres du VIASM pour leurs hospitalité pendant ma visite à Hanoi où l'article a été baptisé. 
\section{Modules instables injectifs}
La lettre $p$ désigne un nombre premier. Dans le cadre de cet article, nous nous intéressons au cas $p=2$. 
\subsubsection*{L'algèbre de Steenrod}
L'algèbre de Steenrod $\a2$ est une algèbre graduée associative engendrée par les $Sq^{k}$ de degré $k\geq 0$ soumis aux relations d'Adem et $Sq^{0}=1$. Serre introduit les notions d'admissible et d'excès \cite{Ser53} pour les opérations de Steenrod. Le monôme $$Sq^{i_{1}}\ldots Sq^{i_{m}}$$
est admissible si $i_{k}\geq 2 i_{k+1}$ pour $2\leq k\leq m$ et $i_{m}\geq 1$. L'excès de cette opération est défini par:
$$e(Sq^{i_{1}}\ldots Sq^{i_{m}}):=i_{1}-i_{2}-\cdots-i_{m}.$$
L'ensemble des monômes admissibles et $Sq^{0}$ forment une base additive de l'algèbre de Steenrod. Milnor \cite{Mil58} montre que $\a2$ a un co-produit naturel qui fait d'elle une algèbre de Hopf et incorpore l'involution de Thom comme la conjugaison. 

Les $Sq^{2^{k}}$ sont indécomposables et on désigne par $\mathcal{A}(n)$ la sous-algèbre de $\a2$ engendrée par les $Sq^{2^{i}},i\leq 2^{n}$. Les $\mathcal{A}(n)$ sont finis d'après \cite[section 5]{Ada58}. En fait, toute sous-algèbre, engendrée par un ensemble fini d'opérations de Steenrod, est finie. Il existe un ensemble minimal de relations définissant $\a2$, concernant seuls les $Sq^{2^{n}}$ \cite{Wal60}.
\begin{align*}
Sq^{2^{i}}Sq^{2^{j}}&\equiv Sq^{2^{j}}Sq^{2^{i}} \textup{ modulo } \mathcal{A}(i-1), \textup{ si } 0\leq j\leq i-2,\\
Sq^{2^{i}}Sq^{2^{i}}&\equiv Sq^{2^{i-1}}Sq^{2^{i}}Sq^{2^{i-1}}+Sq^{2^{i-1}}Sq^{2^{i-1}}Sq^{2^{i}} \textup{ modulo } \mathcal{A}(i-1).
\end{align*}
Pour $n\geq k$, on désigne par $Q^{n}_{k}$ le produit
$$Sq^{2^{k}}Sq^{2^{k+1}}\ldots Sq^{2^{n}}.$$
L'ensemble des monômes $Q^{n_{0}}_{k_{0}}Q^{n_{1}}_{k_{1}}\ldots Q^{n_{i}}_{k_{i}}$
où $(n_{j},k_{j})$ sont en ordre lexicographique
$$(n_{j},k_{j})<(n_{j-1},k_{j-1}) \textup{ pour tout }1\leq j\leq i,$$
forme la base de Wall de l'algèbre de Steenrod \cite{Wal60}.
\subsubsection*{Modules instables}
Un $\a2-$module $M$ est dit instable si l'action de $Sq^{k}$ sur un élément homogène $x$ de degré $n$ est triviale dès que $k>n$. La catégorie des modules instables est désignée par $\u$. 

On désigne par $|-|$ le degré d'un élément. Soit $\mathrm{Sq}_{0}$ l'opération qui, à un élément homogène $x$ d'un module instable, associe l'élément $Sq^{|x|}$. La torsion de Frobenius dans $\u$ est un endofoncteur $\Phi$ qui, à un module instable $M$ associe le module $\Phi M$, concentré en degrés pairs et $\left( \Phi M\right)^{2n}=M^{n}.$ Posant 
\begin{align*}
\lambda_{M}:\Phi M&\to M\\
\Phi x&\mapsto \mathrm{Sq}_{0}x
\end{align*}
alors un module instable $M$ est dit réduit si $\lambda_{M}$ est injectif. Il est dit nilpotent si pour chaque élément $x\in M$, il existe un entier $n_{x}$ tel que $$\mathrm{Sq}_{0}^{n_{x}}x=0.$$ 
On désigne par $\nil$ la sous catégorie des modules nilpotents.

Notons $J(n)$ l'enveloppe injective de la cohomologie réduite $\tilde{\h}^{*}\left(S^{n};\f2\right)$ du sphère $S^{n}$. Les $J(n)$ sont injectifs, caractérisés par $\ho{\u}{M}{J(n)}\cong \left( M^{n}\right)^{*}.$ Ces modules sont finis et donc nilpotents. Dualement, il y a des $F(n)$, instablement et librement engendré par $\imath_{n}$ de degré $n$. Les $F(n)$ sont projectifs, caractérisés par $\ho{\u}{F(n)}{M}\cong M^{n}$. Le module $F(1)$ s'injecte dans la cohomologie $$\h^{*}\z/2\cong \f2[u]$$
donc le générateur de $F(1)$ est souvent désigné par $u$ au lieu de $\imath_{1}$.

Carlsson et Miller ont observé que la cohomologie modulo $p$ d'un $p-$groupe abélien élémentaire est injective dans $\u$ \cite{Car83,Mil84}. Lannes et Zarati ont démontré que le produit tensoriel d'un tel module avec un module de Brown-Gitler reste injectif \cite{LZ86} et cela donne:
\begin{theo}[{\cite[théorème 3.1]{LS89}}]\label{ls89}
Chaque module injectif indécomposable de la catégories $\u$ est un produit tensoriel $L\otimes J(n)$ entre un facteur direct indécomposable de $(\h^{*}B\z/p)^{\otimes d}$ et un module de Brown-Gitler.
Un module instable injectif est la somme directe des modules de ces types.
\end{theo}
Ce théorème signifie qu'un module instable injectif se décompose en somme directe entre un module réduit et un module nilpotent. Pour chaque $n$, on désigne par $x_{n}$ le seul générateur de $\left(J(2^{n})\right)^{1}$. Les modules de Brown-Gitler ont explicités par Miller: 
\begin{theo}[{\cite[théorème 6.1]{Mil84}}]
Pour $p=2$, il y a un isomorphisme d'algèbres bi-graduées: $$\dsum{n\geq 0}{}{J(n)}\xrightarrow{\sim}\f2[x_{n},n\geq 0, ||x_{n}||=(1,2^{n})].$$
L'action de l'algèbre de Steenrod est définie par $Sq^{1}x_{n}=x_{n-1}^{2}$ et par la formule de Cartan:
$$Sq^{m}(xy)=\sum_{i=0}^{m}Sq^{i}(x)Sq^{m-i}y.$$
\end{theo}
\subsection{La réduction du travail}
Posons:
$$
\h_{r}=\frac{F(1)}{\Phi^{r}F(1)}.
$$ 
De manière récursive, on définit  pour $k\geq 2$: $$\lambda_{M}^{k}:=\lambda_{M}\circ \Phi\left(\lambda_{M}^{k-1}\right):\Phi^{k}M\to M.$$ Un module instable $M$ est dit connexe si il est nul en degré $0$. 

\begin{prop}\label{phi}
Soient $M$ un module instable connexe et $r$ un nombre entier. Alors le morphisme $\lambda_{F(1)}^{r}$ induit un isomorphisme 
\begin{equation}\label{eqk1}
\ex{\u}{*}{\Phi^{r}M}{\Phi^{r}F(1)}\cong\ex{\u}{*}{\Phi^{r}M}{F(1)}.
\end{equation}
De plus
\begin{equation}\label{eqk2}
\left(\lambda_{\Phi^{r}M}\right)^{*}\circ \left(\lambda^{r}_{F(1)}\right)_{*} =\left(\lambda^{r+1}_{F(1)}\right)_{*}\circ \ex{\u}{*}{\Phi}{\Phi}.
\end{equation}
\end{prop}
\begin{proof}
L'isomorphisme \eqref{eqk1} se résulte des remarques suivantes:
\begin{enumerate}
\item les groupes $\ex{\u}{i}{\Phi^{r}M}{\h_{r}}$ sont triviaux;
\item le module $\h_{r}$ s'insère dans la suite exacte
$$0\to \Phi^{r}F(1)\to F(1)\to\h_{r}\to 0.$$
\end{enumerate}
L'identité \eqref{eqk2} se découle de ce que le morphisme $\ex{\u}{*}{\lambda_{\Phi^{r}M}}{\Phi^{r}F(1)}$ se factorise à travers $\ex{\u}{*}{\Phi^{r+1}M}{\lambda_{\Phi^{r}F(1)}}$ \textit{via} $\ex{\u}{*}{\Phi}{\Phi}$. 
\end{proof}
A l'aide de la proposition \ref{phi}, on se ramène à étudier la résolution injective minimale du module $F(1)$ et les morphismes $\left(\lambda_{\Phi^{n-1}F(1)}\right)^{*}$.
\begin{coro} Soit $M$ un module instable connexe. Il existe un isomorphisme entre les co-limites:
$$
\mathrm{colim}_{n}\ex{\u}{*}{\Phi^{n}M}{F(1)}\cong \mathrm{colim}_{n}\ex{\u}{*}{\Phi^{n}M}{\Phi^{n}F(1)}.
$$
De plus les $\ex{\u}{*}{\Phi}{\Phi}$ sont des monomorphismes si et seulement si $\left(\lambda_{\Phi^{n-1}M}\right)^{*}$ les sont.
\end{coro}
On désigne par $(I^{\bullet},\partial^{\bullet})$ la résolution injective minimale de $F(1)$. Chaque module $I^{j}$ se scinde en somme directe $R^{j}\oplus N^{j}$: $R^{j}$ est réduit et $N^{j}$ est nilpotent \cite[théorème 3.1]{LS89}. Puisque
$
\ho{\u}{N^{i}}{R^{i+1}}\textup{ est nul pour tout }i\geq 0,
$ les morphismes restreints $\partial^{\bullet}|_{N^{\bullet}}$ font de la suite $\left(N^{r},\partial^{r}|_{N^{r}} \right)^{r\geq 0}$ un sous-complexe de la résolution injective minimale de $F(1)$. Désormais $\partial^{l}|_{N^{l}}:N^{l}\to N^{l+1}$ est noté par $\partial^{l}_{n}$. On peut exprimer les différentielles $\partial^l$ sous forme matricielle:
\begin{equation}\
\left(\begin{matrix}
\partial^{l}_{n} & \omega^{l}\\
0 & \partial^{l}_{r}
\end{matrix}\right):N^{l}\oplus R^{l}\to N^{l+1}\oplus R^{l+1}. 
\end{equation}
Désigne par $\tilde{\Phi}$ l'adjoint à droite du foncteur $\Phi$. Comme les $R^{i}$ sont injectifs réduits, alors  $\tilde{\Phi}R^{i}\cong R^{i}$ \cite[théorème 6.3.4]{Sch94}. On a des isomorphismes:
\begin{align*}
\ho{\u}{\Phi^{n}F(1)}{I^{m}}&\cong \ho{\u}{\Phi^{n}F(1)}{N^{m}\oplus R^{m}}\\
					 		&\cong \left(\tilde{\Phi}^{n}N^{m}\right)^{1}\oplus \left(R^{m}\right)^{1}.
\end{align*}
Par conséquent, il suffit de considérer la partie $\left(N^{\bullet},\partial_{n}^{\bullet}\right)$ et seul le degré $1$ de la partie $\left(R^{\bullet},\partial_{r}^{\bullet}\right)$ pour calculer les groupes $\ex{\u}{*}{\Phi^{n}F(1)}{F(1)}$. 
\subsection{L'homologie du complexe $\bm{(R^{\bullet},\partial^{\bullet}_r)}$}
D'après \cite[I.7]{HLS93} le foncteur $f:\u\to\F_{\omega}$ admet un adjoint à droite que l'on note $m$. En particulier, $f\circ m$ est équivalent au endo-foncteur $id_{\F_{\omega}}$.
\begin{notat} On note $\ell$ le foncteur composé $m\circ f$. Il est appelé foncteur de localisation loin de $\nil$.
\end{notat}
Un module instable $M$ est dit $\nil-$fermé si $\ex{\u}{i}{N}{M}$ est nul pour tout module nilpotent $N$ et pour $i=0,1$.
\begin{prop}[{\cite[I.7]{HLS93}}]
Le module $\ell(M)$ est $\nil-$fermé pour tout module instable $M$. De plus si $M$ est $\nil-$fermé, $M\xrightarrow{\sim}\ell(M)$.
\end{prop}
Puisque les modules injectifs réduits sont $\nil-$fermés, alors
\begin{coro}\label{reduitf1}
On a une identification des complexes: $$
\ell(I^{\bullet},\partial^{\bullet})\cong (R^{\bullet},\partial^{\bullet}_r).
$$ L'homologie du complexe $(R^{\bullet},\partial^{\bullet}_r)$ est isomorphe au dérivé $\ell^{*}(F(1))$.
\end{coro}
La description de $\ell^{*}(F(1))$ que l'on donne dans la suite est une récolte des travaux présentés dans \cite{FLS94} (voir aussi \cite{Sch03}).
\begin{theo}
Étant donnés $M$ dans $\u$ et $F$ dans $\F_{\omega}$, il existe une suite spectrale du premier quadrant 
\begin{equation}\label{Grospec}
\ex{\u}{i}{M}{\ell^j(m(F))}\Rightarrow\ex{\F}{i+j}{f(M)}{F}.
\end{equation}
\end{theo}
\begin{proof}
On considère la paire des foncteurs 
$$
\u\xrightarrow{\ell}\u\xrightarrow{\ho{\u}{M}{-}}\v_{\f2}.
$$
Le foncteur $f$ est exact. Puisque le foncteur $m$ est exact à gauche, le foncteur $\ell$ l'est aussi. Dans la mesure où $m$ et $f$ préservent l'injectivité des objets, il en est de même pour $\ell$. On déduit de la suite spectrale de Grothendieck associée à la paire $\left\{\ell,\ho{\u}{M}{-}\right\}$ qu'il y a une suite spectrale du premier quadrant convergeant vers
\begin{align*}
\rd{i+j}{\ho{\u}{M}{\ell(-)}}&\cong\rd{i+j}{\ho{\u}{f(M)}{f(-)}}\\
&\cong\ex{\F}{i+j}{f(M)}{f(-)}
\end{align*}
dont la deuxième page est 
$$
\rd{i}{\ho{\u}{M}{-}}\left(\rd{j}{\ell}\right)\cong\ex{\u}{i}{M}{\ell^j(-)}.
$$
En l'appliquant au module $m(F)$ on obtient la suite spectrale
$$
E_2^{i,j}\cong\ex{\u}{i}{M}{\ell^j(m(F))}\Rightarrow\ex{\F}{i+j}{f(M)}{F}
$$
d'où le résultat.
\end{proof}
On note $I$ le foncteur d'inclusion $V\mapsto V,$ et $\Gamma^{k}$ le foncteur qui, à un espace vectoriel $V$, associe le groupe des invariants $\left(V^{\otimes k}\right)^{\mathfrak{S}_{k}}$ où le groupe symétrique $\mathfrak{S}_{k}$ agit par permutations sur $V^{\otimes k}.$ Puisque $F(k)$ est projectif, la suite spectrale (\ref{Grospec}) pour $(F(k),I)\in \u\times \F_{\omega}$ s'effondre et donne:
\begin{coro}\label{l1}
Il y a un isomorphisme:
$$\ell^i(m(I))^k\cong \ex{\F}{i}{\Gamma^k}{I}.$$
\end{coro}
Le théorème suivant caractérise des relations entre les $\ell^i(m(I))^k$.
\begin{theo}[\cite{FLS94}]\label{l2}
Les groupes $\ex{\F}{i}{\Gamma^{m}}{I}$ sont triviaux si $m$ n'est pas une puissance de $2$ et: 
$$
\ex{\F}{i}{\Gamma^{2^k}}{I}=\left\{\begin{array}{cl}
\f2 & \text{ si } 2^{k+1}|i,\\
0 & \text{ sinon.}\end{array}\right. 
$$
De plus les morphismes
$$
\ex{\F}{i}{\Gamma^{2^{k-1}}}{I}\xrightarrow{f^{*}\left(Sq^{2^{k-1}}\bullet\right)} \ex{\F}{i}{\Gamma^{2^k}}{I}
$$  
sont des isomorphismes si $2^{k+1}|i$.
\end{theo}
\begin{coro}\label{calcul1}
Soit $i=2^n(2k+1)$. Il y a un isomorphisme de modules instable:$$\ell^i(F(1))\cong \frac{F(1)}{\Phi^n F(1)}.$$
\end{coro}
\begin{proof}
L'action de $Sq^i$ sur $M^{n}$ se détermine par l'action de $Sq^i:F(n+i)\to F(n)$ sur $\ho{\u}{F(n)}{M}$. L'isomorphisme
$$
\ell^i(m(I))^k\cong \ex{\F}{i}{\Gamma^k}{I},
$$
combiné avec le théorème \ref{l2}, montre que, en tant qu'espaces vectoriels gradués: 
$$
\ell^i(F(1))\cong \frac{F(1)}{\Phi^n F(1)}\textup{ si } i=2^n(2k+1).
$$
A cause de l'écart de degrés, la seule opération agissant non-trivialement sur  $\left(\ell^i(F(1))\right)^{2^k}$ est $Sq^{2^k}$. A travers la suite spectrale
$$
\ho{\u}{F(k)}{\ell^i(F(1))}\Rightarrow \ex{\F}{i}{\Gamma^k}{I},
$$
cette opération devient 
$$
\ex{\F}{i}{\Gamma^{2^k}}{I}\xrightarrow{f^{*}\left(Sq^{2^k}\bullet\right)}\ex{\F}{i}{\Gamma^{2^{k+1}}}{I}
$$
et est donc non-triviale. Ce morphisme est en fait induit par le Verchiebung $\Gamma^{2^{k+1}}\to\Gamma^{2^k}$, dual au morphisme de Frobenius $S^{2^{k}}\to S^{2^{k+1}}$ (voir \cite{Hai}). Il s'ensuit que si $i=2^n(2k+1)$
$$
\ell^i(F(1))\cong \frac{F(1)}{\Phi^n F(1)}
$$
en tant que modules instables.
\end{proof} 
En particulier, d'après \cite[théorème 7.3]{FLS94}:
\begin{theo}\label{fls94}
Le produit de Yoneda fait de $\ex{\F}{*}{I}{I}$ est une $\f2-$algèbre commutative. Elle est engendrée par les classes $e_{n}\in\ex{\F}{2^{n+1}}{I}{I}$ et admet la présentation suivante
$$
\ex{\F}{*}{I}{I}\cong\frac{\f2\left[e_0,e_1,\ldots,e_n,\ldots\right]}{\langle e_n^2,n\in\n\rangle},
$$
$\langle e_n^2,n\in\n\rangle$ désignant l'idéal engendré par les puissances $2-$ièmes.
\end{theo}

\subsubsection*{Contre-exemples} 
On aimerait que le morphisme $\ex{\u}{*}{\Phi}{\Phi}$ soit injectif en général. Cependant:
\begin{conex}
Il existe un entier $i$ tel que le morphisme $\ex{\u}{5}{\Phi}{\Phi}$ $$\ex{\u}{5}{\Phi^{i}(F(1)\otimes F(1))}{\Phi^{i}(F(1))}\to \ex{\u}{5}{\Phi^{i+1}(F(1)\otimes F(1))}{\Phi^{i+1}(F(1))}$$
n'est pas injectif.
\end{conex}
On va maintenant justifier ce contre-exemple. Dans un premier temps, on montre que $$\mathrm{colim}_{n}\ex{\u}{*}{\Phi^{n}(F(1)\otimes F(1))}{F(1)}=0.$$ Par contre, dans un deuxième temps on vérifie que $$\ex{\u}{5}{F(1)\otimes F(1)}{F(1)}\cong\f2.$$ Alors il existe $i$ tel que $\left(\lambda_{\Phi^{i}(F(1)\otimes F(1))}\right)^{*}$ n'est pas injectif. Au cas contraire:
$$
\f2\subset\mathrm{colim}_{n}\ex{\u}{*}{\Phi^{n}(F(1)\otimes F(1))}{F(1)},
$$ 
ce qui contredit la trivialité de $\mathrm{colim}_{n}\ex{\u}{*}{\Phi^{n}(F(1)\otimes F(1))}{F(1)}$. 

Selon \cite{HLS93,CS15}, on sait calculer $\mathrm{colim}_{n}\ex{\u}{*}{\Phi^{n}M}{F(1)}$:
\begin{theo} Soit $n$ un entier. Le morphisme
$$
\ex{\u}{*}{\Phi^{n}M}{F(1)}\to \ex{\F}{*}{f(M)}{I},
$$ 
naturel en $M$, induit un isomorphisme
$$
\mathrm{colim}_{n}\ex{\u}{*}{\Phi^{n}M}{F(1)}\cong \ex{\F}{*}{f(M)}{I}
$$
naturel en $M$.
\end{theo} 
De plus, d'après \cite{FLS94}
$$
\ex{\F}{*}{f(F(1)\otimes F(1))}{I}=\ex{\F}{*}{I\otimes I}{I}=0,
$$ 
alors: 
$$
\mathrm{colim}_{n}\ex{\u}{*}{\Phi^{n}(F(1)\otimes F(1))}{F(1)}=0.
$$ 
Les groupes $\ex{\u}{*}{\Phi^{n}(F(1)\otimes F(1))}{F(1)}$ doivent être nuls si on suppose que $\left(\lambda_{\Phi^{n-1}(F(1)\otimes F(1))}\right)^{*}$ est injectif. Cependant:
\begin{lemm}\label{contre1}
On a un isomorphisme
$$
\ex{\u}{5}{F(1)\otimes F(1)}{F(1)}\cong\f2.
$$
\end{lemm}
\begin{proof}
On désigne par $\Lambda^{2}$ la $2-$ième puissance extérieure. Il résulte de la résolution injective minimale de $\Lambda^{2}(F(1))$ \cite[corollaire 1.1.4.17]{Cuo14b} qu'on a un isomorphisme: $$\ex{\u}{5}{\Lambda^2(F(1))}{F(1)}\cong\f2.$$
Le lemme découle de la suite exacte longue associée à la suite exacte courte
$$
0\to F(2)\to F(1)\otimes F(1)\to \Lambda^2(F(1))\to 0.
$$
\end{proof}
Un contre-exemple concernant les modules nilpotents est donné ci-dessous. 
\begin{conex}
Le morphisme 
$$\ex{\u}{3}{\Sigma\f2}{F(1)}\to\ex{\u}{3}{\Sigma^{2}\f2}{\Phi F(1)}$$
n'est pas injectif.
\end{conex}  
\section{Sur la résolution injective minimale de $F(1)$}
D'après la section 2, on désigne par $(I^{\bullet},\partial^{\bullet})=(N^{\bullet}\oplus R^{\bullet},\partial^{\bullet})$ la résolution injective minimale de $F(1)$, $R^{\bullet}$ et $N^{\bullet}$ désignant la partie réduite et la partie nilpotente respectivement. Les morphismes $\partial^l$ s'écrivent sous forme matricielle:
\begin{equation}\label{omegal}
\left(\begin{matrix}
\partial^{l}_{n} & \omega^{l}\\
0 & \partial^{l}_{r}
\end{matrix}\right):N^{l}\oplus R^{l}\to N^{l+1}\oplus R^{l+1}. 
\end{equation}
 La minimalité de la résolution signifie que $I^{0}$ est l'enveloppe injective de $F(1)$, $I^{1}$ est celle du quotient $I^{0}/F(1)$  et $I^{j}$ est l'enveloppe injective du conoyau $\cke{\partial^{j-2}}$ pour $j\geq 2$.
\subsection{La partie réduite de la résolution}
Ce paragraphe est consacré pour étudier la partie $\left(R^{\bullet},\partial_{r}^{\bullet}\right)$ de la résolution injective minimale de $F(1)$. En particulier, chaque module $R^{j}$ sera calculé en degré $1$.

L'enveloppe injective de $F(1)$ est $\tilde{\h}^{*}\z/2$. Nous allons montrer que $R^{k}$ contient un seul facteur direct isomorphe à $\tilde{\h}^{*}\z/2$ si $k$ est pair et n'en contient aucun sinon. Pour ce fait, remarquons que d'une part $f(F(1))=I$ est simple dans $\F$ et d'autre part, $f$ préserve la minimalité des résolution. Alors le nombre de facteurs isomorphes à $\tilde{\h}^{*}\z/2$  de $R^{k},$  est la dimension sur $\ho{\F}{I}{I}$ du groupe $\ex{\F}{k}{I}{I}.$ Il découle du théorème \ref{l2} que:
\begin{prop}
Chaque $R^{k}$ contient un seul facteur direct isomorphe à $\tilde{\h}^{*}\z/2$ si $k$ est pair, et n'en contient aucun sinon.
\end{prop}
A part $\tilde{\h}^{*}\z/2$ et $\f2$, les autres modules instables injectifs indécomposables sont
$1-$connexes \cite[section 4.4]{Sch94}. Alors les $R^{2l+1}$ sont $1-$connexes et
$$
\left(R^{2l}\right)^{1}\cong\f2.
$$
Par abus de notation on note $u$ le seul générateur de degré $1$ de $R^{2k}$. L'exactitude de la résolution injective de $F(1)$ implique que $\partial_{r}^{2l}(u)$ est trivial alors que $\omega^{2l}(u)$ ne l'est pas. Il en découle que chaque terme $N^{2l+1}$ contient un facteur direct $J(2^{n_l})$ tel que $$\omega^{2l}(u)=x_{n_l}.$$

Dans le reste de cet article, on va montrer que dans plusieurs cas intéressants, le terme $N^{2l+1}$
contient un seul facteur direct de type $J(2^{n_l})$ alors que $N^{2l}$ n'en contient aucun et cela
est suffisant pour calculer le groupe $\ex{\u}{2l}{\Phi^{r}F(1)}{F(1)}$. En effet, on va montrer que
dans ces cas:
\begin{prop}\label{calculessentiel}
Soit $l$ un entier. Si il existe un entier $n_{l}$ tel que  $$N^{2l+1}=J(2^{n_{l}})\oplus
\bigoplus_{\alpha}J\left( 2^{m_{\alpha}}(2t_{\alpha}+1)\right)$$ et que $N^{2l}$ ne contient
aucun facteur direct de type $J(2^{n})$ alors on a:
$$\
\ex{\u}{2l}{\Phi^{k}F(1)}{F(1)}\cong\left\{\begin{array}{cl}
0 &\text{ si } k\leq n_{l},\\
\f2&\text{ si } k> n_{l}.
\end{array}\right.
$$
De plus, les morphismes $$\ex{\u}{2l}{\Phi^{k}F(1)}{F(1)}\to \ex{\u}{2l}{\Phi^{k+1}F(1)}{F(1)}$$
sont injectifs.
\end{prop}
\begin{proof}
Puisque $N^{2l+1}$ contient une seule copie de $J\left(2^{n_{l}}\right)$, alors si $k\leq n_{l}$:
\begin{align*}
\ho{\u}{\Phi^{k}F(1)}{N^{2l+1}}&\cong\ho{\u}{\Phi^{k}F(1)}{J\left(2^{n_{l}}\right)}
\cong \langle x_{n_{l}-k}^{2^k}\rangle.
\end{align*}
Parce que d'une part $$\ell^{2l}(F(1))\cong \langle
u,u^2,\ldots,u^{2^{\left[\log_{2}l\right]-1}}\left|u\in  I^{2l}\right.\rangle$$ et d'autre part,
$\partial^{2l}(u^{2^k})=x_{n_{l}-k}^{2^k}$ alors le morphisme 
$$\h^{2l}\left(R^{\bullet}\right)\to \h^{2l+1}\left(N^{\bullet}\right)$$
est un isomorphisme $\f2\xrightarrow{\sim}\f2$. Il découle de la suite exacte longue associée à la
suite exacte courte des complexes 
$$0\to N^{\bullet}\to I^{\bullet}\to R^{\bullet}\to 0$$ 
que 
$$
\ex{\u}{2l}{\Phi^{k}F(1)}{F(1)}\cong 0.
$$
Si $k\geq n_{l}+1$, il résulte de la trivialité de 
$
\ho{\u}{\Phi^{k}F(1)}{N^{2l+1}}
$ 
que
$$
\ex{\u}{2l+1}{\Phi^{k}F(1)}{F(1)}\cong 0
$$ 
et que
\begin{align*}
\ex{\u}{2l}{\Phi^{k}F(1)}{F(1)}\cong\left(\ell^{2l}(F(1))\right)^{1}\cong\f2.
\end{align*}
De plus, parce que 
$$\ho{\u}{\frac{\Phi^{k}F(1)}{\Phi^{k+1}F(1)}}{I^{2l}}\cong
\ho{\u}{\frac{\Phi^{k}F(1)}{\Phi^{k+1}F(1)}}{N^{2l}\oplus R^{2l}}\cong 0$$
alors $\ex{\u}{2l}{\Phi^{k}F(1)/\Phi^{k+1}F(1)}{F(1)}$ est trivial et 
$$\ex{\u}{2l}{\Phi^{k}F(1)}{F(1)}\to \ex{\u}{2l}{\Phi^{k+1}F(1)}{F(1)}$$
est injectif.
\end{proof}
\subsection{La partie nilpotente de la résolution}
Puisque d'une part, les $\h^{k}\left(R^{\bullet},\partial_{r}^{\bullet}\right)$ sont finies et
d'autre part
$$
\h^{k}\left(R^{\bullet},\partial_{r}^{\bullet}\right)\cong \h^{k+1}\left(N^{\bullet},\partial_{n}^{\bullet}\right),
$$
alors les $\h^{k+1}\left(N^{\bullet},\partial_{n}^{\bullet}\right)$ sont finies. Dans ce paragraphe, on va montrer que les modules $N^{l}$ sont aussi finis.

Remarquons que si un module instable est nilpotent ou réduit, il en est de même pour son enveloppe injective. Un module instable peut se calculer comme l'extension d'un module réduit par le plus grand sous-module nilpotent. Le lemme suivant explique comment on forme l'enveloppe injective de l'extension à partir de la partie nilpotente et celle qui est réduite. La vérification de ce lemme est simple et est laissée aux lecteurs.
\begin{lemm}\label{rs5}
Étant donné une suite exacte courte 
$$
0\to N\to M\to R\to 0,$$
$N$ désignant un module nilpotent et $R$ désignant un module réduit, l'enveloppe injective de $M$ est la somme directe de l'enveloppe injective de $R$ et celle de $N$.
\end{lemm}
Alors, pour déterminer $N^{j}$ il faut calculer les sous-modules nilpotents les plus grands des $\cke{\partial^{j-2}}$. On se place dans une situation générale des catégories abéliennes. La proposition suivante est classique et est laissée aux lecteurs.
\begin{prop}
Soient $\mathcal{C}$ une catégorie abélienne et $\left( I^{j},\partial^{j}\right)^{j\geq
0}$ une
suite exacte dans $\mathcal{C}$. On désigne par $\left( N^{j},\partial_{n}^{j}\right)^{j\geq 0}$ un
sous-complexe de cette suite. On note $\left( R^{j},\partial_{r}^{j}\right)^{j\geq 0}$ le complexe
quotient $\left( I^{j}/N^{j}\right)^{j\geq 0}.$ Alors, le noyau $M^{j}$ du composé
$$\cke{\partial^{j}}\to\cke{\partial^{j}_{r}}\to\frac{R^{j+1}}{\ke{\partial_{r}^{j+1}}}$$
s'insère dans une suite exacte:
$$ 0\to \h^{j}\left( R^{\bullet},\partial_{r}^{\bullet}\right)\to
\cke{\partial_{n}^{j}}\to M^{j}\to \h^{j+1}\left( R^{\bullet},\partial_{r}^{\bullet}\right)\to
0.$$
\end{prop}
Cette proposition applique à la résolution injective minimale de $F(1)$.
\begin{prop}
Le noyau $\mathscr{M}^{j}$ du composé  
$$\cke{\partial^{j}}\to\cke{\partial^{j}_{r}}\to\frac{R^{j+1}}{\ke{\partial_{r}^{j+1}}}$$
s'insère dans une suite exacte:
$$ 0\to \ell^{j}\left(F(1)\right)\to
\cke{\partial_{n}^{j}}\to \mathscr{M}^{j}\to \ell^{j+1}\left(
F(1)\right)\to
0.$$
De plus, $N^{j+2}$ est l'enveloppe injective de $\mathscr{M}^{j}.$
\end{prop}
\begin{proof}
La proposition résulte des trois points suivants:
\begin{enumerate}
\item le lemme \ref{rs5};
\item les $R^{j+1}/\ke{\partial_{r}^{j+1}}$ sont réduits;
\item les  $\ell^{j}\left(F(1)\right)$ et $\cke{\partial_{n}^{j}}$ sont nilpotents.
\end{enumerate}
\end{proof}
Comme $F(1)$ est réduit, $I^{0}$ l'est aussi et il suit que $N^0=0$. Alors, par une récurrence
simple on obtient:
\begin{coro}\label{lemmeimp}
Dans la résolution injective minimale de $F(1)$, les $N^i$ sont finis.
\end{coro}

Si $M$ est un module instable fini, on désigne par $\mathrm{d}(M)$ le plus grand degré tel que $M^{n}$ ne soit pas trivial. Pour exemple, $\mathrm{d}(J(n))=n.$ 
Le lemme suivant mesure les $\mathrm{d}(N^{i})$.
\begin{lemm}\label{rs4}
Si $n>k+1\geq 2$ alors:
\begin{longtable}{|c||*{11}{c|}}
\hline
\rowcolor{Gray}$m$&$2^{n}-2$&$2^{n}-1$&$2^{n}$&$2^{n}+1$&$2^n-2^k-1$&$2^n-2^k$&$2^n-2^k+1$\\
\hline
$N^{m}$& $0$ & $J(1)$& $0$& $J\left(2^{n-1}\right)$& $J(2)$& $0$&$J\left(2^{k-1}\right)$\\
\hline
\end{longtable}
Si $l$ est un entier tel que $2^{k-1}-1\geq l\geq 2$, alors:
$$
\de{N^{2^n-2^k+l}}= 2^{k-1}-l+1\text{ et } \left(N^{2^n-2^k+l}\right)^{2^{k-1}-l+1}\cong\f2.
$$
\definecolor{xdxdff}{rgb}{0.49019607843137253,0.49019607843137253,1.0}
\definecolor{qqqqff}{rgb}{0.0,0.0,1.0}
\definecolor{cqcqcq}{rgb}{0.7529411764705882,0.7529411764705882,0.7529411764705882}
\begin{tikzpicture}[line cap=round,line join=round,>=triangle 45,x=0.27cm,y=0.4cm]
\draw [color=cqcqcq,dash pattern=on 1pt off 1pt, xstep=0.54cm,ystep=0.4cm] (-0.5,-0.5) grid (36.0,17.0);
\draw[->,color=black] (-0.5,0.0) -- (36.0,0.0);
\foreach \x in {0,1,2,3}
\draw[shift={(\x+1,0)},color=black] (0pt,2pt) -- (0pt,-2pt) node[below] {\footnotesize $t_\x$};
\foreach \x in {0,1,2,3,4,5,6}
\draw[shift={(\x+13,0)},color=black] (0pt,2pt) -- (0pt,-2pt) node[below] {\footnotesize $r_\x$};
\foreach \x in {0,1,2,3}
\draw[shift={(\x+21,0)},color=black] (0pt,2pt) -- (0pt,-2pt) node[below] {\footnotesize $s_\x$};
\foreach \x in {0,1,2,3,4,5,6}
\draw[shift={(\x+29,0)},color=black] (0pt,2pt) -- (0pt,-2pt) node[below] {\footnotesize $q_\x$};
\draw[->,color=black] (0.0,-0.5) -- (0.0,17.0);
\foreach \y in {0}
\draw[shift={(0,\y+16)},color=black] (2pt,0pt) -- (-2pt,0pt) node[left] {\footnotesize $2^{k-1}$};
\foreach \y in {1}
\draw[shift={(0,\y+14)},color=black] (2pt,0pt) -- (-2pt,0pt) node[left] {\footnotesize $2^{k-1}-1$};
\foreach \y in {2}
\draw[shift={(0,\y+12)},color=black] (2pt,0pt) -- (-2pt,0pt) node[left] {\footnotesize $2^{k-1}-2$};
\foreach \y in {3}
\draw[shift={(0,\y+10)},color=black] (2pt,0pt) -- (-2pt,0pt) node[left] {\footnotesize $2^{k-1}-3$};
\foreach \y in {-1}
\draw[shift={(0,\y+18)},color=black] (0pt,0pt) -- (0pt,0pt) node[left] {\footnotesize $d(N^{j})$};
\foreach \y in {8}
\draw[shift={(0,\y+0)},color=black] (2pt,0pt) -- (-2pt,0pt) node[left] {\footnotesize $2^{k-2}$};
\foreach \y in {9}
\draw[shift={(0,\y-2)},color=black] (2pt,0pt) -- (-2pt,0pt) node[left] {\footnotesize $2^{k-2}-1$};
\foreach \y in {10}
\draw[shift={(0,\y-4)},color=black] (2pt,0pt) -- (-2pt,0pt) node[left] {\footnotesize $2^{k-2}-2$};
\foreach \y in {0,1,2,3,4}
\draw[shift={(0,\y)},color=black] (2pt,0pt) -- (-2pt,0pt) node[left] {\footnotesize $\y$};
\draw[color=black] (36,0) node[right] {\footnotesize $j$};
\clip(-1,-1) rectangle (36.0,17.0);
\draw (1.0,16.0)-- (2.0,15.0);
\draw (2.0,15.0)-- (3.0,14.0);
\draw (3.0,14.0)-- (4.0,13.0);
\draw [dash pattern=on 5pt off 5pt] (4.0,13.0)-- (13.0,4.0);
\draw (13.0,4.0)-- (14.0,3.0);
\draw (14.0,3.0)-- (15.0,2.0);
\draw (17.0,8.0)-- (18.0,7.0);
\draw (18.0,7.0)-- (19.0,6.0);
\draw [dash pattern=on 5pt off 5pt] (19.0,6.0)-- (21.0,4.0);
\draw (21.0,4.0)-- (22.0,3.0);
\draw (22.0,3.0)-- (23.0,2.0);
\draw (29.0,4.0)-- (30.0,3.0);
\draw (30.0,3.0)-- (31.0,2.0);
\draw [dash pattern=on 1pt off 1pt] (4.5,-0.5)-- (12.5,-0.5);
\draw [dash pattern=on 1pt off 1pt] (24.6,-0.5)-- (28.4,-0.5);
\draw [dash pattern=on 1pt off 1pt] (-0.8,4.4)-- (-0.8,5.6);
\draw [dash pattern=on 1pt off 1pt] (-0.8,8.4)-- (-0.8,12.6);
\draw [dash pattern=on 2pt off 2pt] (1.0,16.0)-- (1.0,0.0);
\draw [dash pattern=on 2pt off 2pt] (2.0,15.0)-- (2.0,0.0);
\draw [dash pattern=on 2pt off 2pt] (3.0,14.0)-- (3.0,0.0);
\draw [dash pattern=on 2pt off 2pt] (4.0,13.0)-- (4.0,0.0);
\draw [dash pattern=on 2pt off 2pt] (13.0,4.0)-- (13.0,0.0);
\draw [dash pattern=on 2pt off 2pt] (14.0,3.0)-- (14.0,0.0);
\draw [dash pattern=on 2pt off 2pt] (15.0,2.0)-- (15.0,0.0);
\draw [dash pattern=on 2pt off 2pt] (17.0,8.0)-- (17.0,0.0);
\draw [dash pattern=on 2pt off 2pt] (18.0,0.0)-- (18.0,7.0);
\draw [dash pattern=on 2pt off 2pt] (19.0,6.0)-- (19.0,0.0);
\draw [dash pattern=on 2pt off 2pt] (21.0,4.0)-- (21.0,0.0);
\draw [dash pattern=on 2pt off 2pt] (22.0,3.0)-- (22.0,0.0);
\draw [dash pattern=on 2pt off 2pt] (23.0,2.0)-- (23.0,0.0);
\draw [dash pattern=on 2pt off 2pt] (29.0,4.0)-- (29.0,0.0);
\draw [dash pattern=on 2pt off 2pt] (30.0,3.0)-- (30.0,0.0);
\draw [dash pattern=on 2pt off 2pt] (31.0,2.0)-- (31.0,0.0);
\draw [dash pattern=on 2pt off 2pt] (33.0,2.0)-- (33.0,0.0);
\draw [dash pattern=on 2pt off 2pt] (35.0,1.0)-- (35.0,0.0);
\begin{scriptsize}
\node[rounded corners,fill=white,text width=2.85cm,align=left] (exp1) at (30,14) [draw] {$t_i:=2^n-2^k+1+i$\\$r_i:=2^n-2^{k-1}-3+i$\\$s_i:=2^n-2^{k-2}-3+i$\\$q_i:=2^{n+1}-7+i$};
\draw [fill=qqqqff] (1.0,16.0) circle (1.5pt);
\draw [fill=qqqqff] (2.0,15.0) circle (1.5pt);
\draw [fill=qqqqff] (3.0,14.0) circle (1.5pt);
\draw [fill=qqqqff] (4.0,13.0) circle (1.5pt);
\draw [fill=qqqqff] (13.0,4.0) circle (1.5pt);
\draw [fill=qqqqff] (14.0,3.0) circle (1.5pt);
\draw [fill=qqqqff] (15.0,2.0) circle (1.5pt);
\draw [fill=qqqqff] (17.0,8.0) circle (1.5pt);
\draw [fill=qqqqff] (18.0,7.0) circle (1.5pt);
\draw [fill=qqqqff] (19.0,6.0) circle (1.5pt);
\draw [fill=qqqqff] (21.0,4.0) circle (1.5pt);
\draw [fill=qqqqff] (22.0,3.0) circle (1.5pt);
\draw [fill=qqqqff] (23.0,2.0) circle (1.5pt);
\draw [fill=qqqqff] (29.0,4.0) circle (1.5pt);
\draw [fill=qqqqff] (30.0,3.0) circle (1.5pt);
\draw [fill=qqqqff] (31.0,2.0) circle (1.5pt);
\draw [fill=qqqqff] (33.0,2.0) circle (1.5pt);
\draw [fill=qqqqff] (35.0,1.0) circle (1.5pt);
\draw [fill=qqqqff] (16.0,0.0) circle (1.5pt);
\draw [fill=qqqqff] (24.0,0.0) circle (1.5pt);
\draw [fill=qqqqff] (32.0,0.0) circle (1.5pt);
\draw [fill=qqqqff] (34.0,0.0) circle (1.5pt);
\end{scriptsize}
\end{tikzpicture}
\end{lemm}
\begin{proof}
Il résulte des corollaires \ref{calcul1} et \ref{lemmeimp} 
que
\begin{longtable}{|c||*{11}{c|}}
\hline
\rowcolor{Gray}\makebox[0.3em]{$m$}&\makebox[3em]{$0$}& \makebox[3em]{$1$}& \makebox[3em]{$2$}&
\makebox[3em]{$3$}& \makebox[3em]{$4$}& \makebox[3em]{$5$}\\
\hline
$N^{m}$& $0$ & $0$& $0$& $J(1)$& $0$& $J(2)$\\\hline
\end{longtable}
Supposons que le lemme est vrai pour tout $n\leq q-1$. On vérifie le cas $n=q$. Par hypothèse de
récurrence, 
\begin{longtable}{|c||*{11}{c|}}
\hline
\rowcolor{Gray}$m$&$2^{q}-4$&$2^{q}-3$&$2^{q}-2$&$2^{q}-1$&$2^{q}$&$2^{q}+1$\\
\hline
$N^{m}$& $0$ & $J(2)$& $0$& $J\left(1\right)$& $0$& $J\left(2^{q-1}\right)$\\
\hline
\end{longtable}
Parce que, d'une part, 
$$\mathrm{d}\left(\mathscr{M}^{2^{q}}\right)=\mathrm{d}\left(\frac{\cke{\partial^{2^{q}}_{
n}}}{
\ell^{2^{q}}(F(1))}\right), $$
et d'autre part la suite 
$$N^{2^{q}+1}\to N^{2^{q}+2}\to N^{2^{q}+3}$$
est exacte en degré $2^{q}-1$ alors 
$$
\de{N^{2^q+2}}= 2^{q-1}-1\text{ et } \left(N^{2^q-1}\right)^{2^{q-1}-1}\cong\f2.
$$
De manière analogue on obtient des égalités pour $2\leq t\leq 2^{q-1}-1$:
$$
\de{N^{2^q+t}}= 2^{q-1}-t+1\text{ et } \left(N^{2^q+t}\right)^{2^{q-1}-t+1}\cong\f2.
$$
Le noyau $\mathscr{M}^{2^{q}+2^{q-1}-3}$ s'insère dans la suite exacte courte:
$$0\to \Sigma^{2}\f2\to \mathscr{M}^{2^{q}+2^{q-1}-3}\to J(1)\to 0.$$
Il découle (voir \cite{Cuo14b}) de la résolution projective minimale de $\Sigma\f2$ que
$$\ex{\u}{2^{q}+2^{q-1}-1}{\Sigma\f2}{F(1)}=0.$$ 
Alors $N^{2^{q}+2^{q-1}-1}$ ne contient aucune copie de $J(1)$ et donc 
$$N^{2^{q}+2^{q-1}-1}\cong J(2).$$
Une récurrence simple sur $s$ tel que $q>s\geq 2$ montre que:
$$
N^{2^{q+1}-2^s-1}\cong J(2),\quad N^{2^{q+1}-2^s}\cong 0,\quad N^{2^{q+1}-2^s+1}\cong J(2^{s-1}),$$
$$
\de{N^{2^{q+1}-2^s+r}}= 2^{s-1}-r+1\text{ et } \left(N^{2^{q+1}-2^s+r}\right)^{2^{s-1}-r+1}\cong\f2,
$$
pour $2\leq r\leq 2^{s-1}-1$. Il en découle que $N^{2^{q+1}-2}\cong 0$ et donc
$N^{2^{q+1}-1}\cong J(1)$. Le lemme en résulte.
\end{proof}
Comme une conséquence du lemme \ref{rs4}, $\mathscr{M}^{2k+1}$ est une extension triviale de
$\h^{2k+2}(R^{\bullet},\partial_{r}^{\bullet})$ par $\cke{\partial_n^{2k+1}}$ si $k$ est de la
forme $2^{n}-2^{m}-1$. On montre en fait qu'elles sont les seules
extensions triviales. Avant de détailler cette classification des extensions dans la section 5, on étudie un phénomène de périodicité de la partie nilpotente de la résolution injective de $F(1).$
\section{Un phénomène de périodicité}\label{perif1}
On rappelle que $\ex{\F}{*}{I}{I}$ est engendrée par $e_{n}\in
\ex{\F}{2^{n+1}}{I}{I}$ en tant qu'algèbre (voir le théorème \ref{fls94}).
\begin{notat}\label{yoneda1} Soient deux entiers $n>k\geq 2$. On désigne: $$e(n,k)=e_{n-2}e_{n-3}\ldots e_{k}.$$
\end{notat}
Le cup-produit avec $e(n,k)$ induit un isomorphisme \cite[proposition 7.2]{FLS94}
\begin{equation}\label{perio1}
\ex{\F}{0}{I}{I}\xrightarrow{\smile e(n,k)}\ex{\F}{2^n-2^{k+1}}{I}{I}.
\end{equation}

On désigne par $\gamma\in\ho{\F}{I}{f\left(R^{0}\right)}$ le
morphisme qui représente l'unité $1$ et par
$\delta\in\ho{\F}{I}{f\left(R^{2^n-2^{k+1}}\right)}$ celui qui représente $e(n,k)$. Puisque
$\left(f(R^{*}),f(\partial^{*})\right)$ est une résolution injective dans $\F$ du foncteur $I$,
l'isomorphisme (\ref{perio1}) signifie que $\delta$ se factorise à travers $\gamma$ \textit{via} un
morphisme $$\gamma_0:f(R^0)\to f(R^{2^n-2^{k+1}}).$$
Grâce à l'exactitude de la suite 
$$I\hookrightarrow f(R^0)\to f(R^1)\to\cdots\to f(R^n)\to\cdots$$
et à l'injectivité des foncteurs $f(R^i)$, on obtient un morphisme de complexes:
$$\gamma_{\bullet}:f(R^{\bullet})\to f\left(R^{2^{n}-2^{k+1}+\bullet}\right)$$
tel que les morphismes induits en cohomologie sont les cup-produits avec $e(n,k)$. On rappelle que
le foncteur $f:\u\to\F$ admet un adjoint à droite qu'on désigne par $m$. Puisque les modules $R^{i}$
sont $\nil-$fermés, alors 
$$\left(m(\gamma_{\bullet}) \right):R^{\bullet}\to
R^{\bullet+2^{n}-2^{k}}$$ 
est un morphisme de complexes. Désignons par 
$\alpha^{i}$ le morphisme $m(\gamma_{i})$. 
Alors pour $0\leq t\leq 2^{k-1}-1$, on a 
\begin{align*}
\alpha^{2^{k}+2t}(u)=u
\end{align*}
et donc le morphisme $\alpha^{2^{k}+2t}$ induit un isomorphisme
$$\ell^{2^k+2t}(F(1))\overset{\sim}{\to}\ell^{2^n-2^k+2t}(F(1)).$$
\begin{theo}[Périodicité]\label{princh42}
Étant donné $n>k\geq 2$ on a:
\begin{align*}
\mathscr{M}^{2^n-2^k+t-2}&\cong \mathscr{M}^{2^k+t-2},\\
N^{2^n-2^k+t}&\cong N^{2^k+t},
\end{align*}
pour tout $0\leq t\leq 2^{k}-1$.
\end{theo}
\begin{proof}
Il suit du lemme \ref{rs4} que $N^{2^n-2^k}$ et $N^{2^{k}}$ sont triviaux.
Alors 
\begin{align*}
\cke{\partial_{r}^{2^{k}}}&\cong\cke{\partial^{2^{k}}}\\
\cke{\partial_{r}^{2^{n}-2^{k}}}&\cong\cke{\partial^{2^{n}-2^{k}}}.
\end{align*}
Puisque $I^{2^{n}-2^{k}+1}$ est injectif, il existe $\beta^{2^{k}+1}:I^{2^{k}+1}\to I^{2^{n}-2^{k}+1}$ qui fait commuter le diagramme
\begin{equation}\label{existe13}
\xymatrixrowsep{1.5pc}\xymatrixcolsep{3pc}\xymatrix{
I^{2^{k}}\ar[r]^{}\ar[d]^{\alpha^{2^{k}}}&
\cke{\partial^{2^{k}-1}}\ar[d]^{\alpha^{2^{k}}}\ar@{^{(}->}[r]&
I^{2^{k}+1}\ar@{-->}[d]^{\exists\beta^{2^{k}+1}} \\
I^{2^n-2^{k}}\ar[r]_{}& \cke{\partial^{2^{n}-2^{k}-1}}\ar@{^{(}->}[r]& I^{2^{n}-2^{k}+1}
}\end{equation}
Parce que, d'une part 
$$
N^{2^{k}+1}\cong J\left(2^{k-1}\right)\cong N^{2^{n}-2^{k}+1}
$$
et d'autre part
$$
\pa{2^{k}}{u}=x_{k-1}=\pa{2^{n}-2^{k}}{u}
$$
alors la restriction de $\beta^{2^{k}+1}$ sur $N^{2^{k}+1}$ est un automorphisme de
$J\left(2^{k-1}\right)$. Puisque les $I^{j}$ sont injectifs, le diagramme \eqref{existe13} montre
qu'il existe un morphisme de complexes:
$$\beta^{2^{k}+\bullet}:I^{2^{k}+\bullet}\to I^{2^{n}-2^{k}+\bullet}.$$
Supposons que $\beta^{2^{k}+t}$ entraîne 
\begin{align*}
\mathscr{M}^{2^n-2^k+t-2}&\cong \mathscr{M}^{2^k+t-2},\\
N^{2^n-2^k+t}&\cong N^{2^k+t},
\end{align*}
pour tout $1\leq t< m< 2^{k-1}-1$. On passe au cas
$t=m$. Par hypothèse de récurrence, $\beta^{2^{k}+m}$  induit un diagramme commutatif
$$\xymatrixcolsep{1pc}\xymatrixrowsep{1.5pc}\xymatrix{
\ell^{2^k+m-2}(F(1))\ar@{=}[d]\ar@{^{(}->}[r]&\cke{\partial_n^{2^{k}+m-2}}\ar@{=}[d]\ar[r]&\mathscr{M}^{2^{k}+m-2}\ar[d]\ar@{->>}[r]&\ell^{2^k+m-1}(F(1))\ar@{=}[d]\\
\ell^{s-2}(F(1))\ar@{^{(}->}[r]&\cke{\partial_n^{\alpha-2}}\ar[r]&\mathscr{M}^{\alpha-2}\ar@{->>}[r]
&\ell^{\alpha-1}(F(1))
}$$
où $s=2^n-2^k+m$. On peut conclure la récurrence et le théorème en découle.
\end{proof}

\section{Sur les extensions $\mathscr{M}^{2k+1}$}
Le cœur de cette section est le théorème suivant:
\begin{theo}\label{exnontriv1}
Les extensions $\mathscr{M}^{2k+1}$ sont triviales si et seulement si $k$ est de la forme
$2^{n}-2^{m}-1,$ où $n>m$.
\end{theo}
Afin de démontrer ce théorème, on introduit ci-dessous une interprétation de la non-trivialité de
l'extension
$\mathscr{M}^{k}$. 
\begin{prop}
Soit $k=2^{r}(2m+1)-1,r\geq 1$. Alors l'extension 
$$\mathscr{M}^{k}\in\ex{\u}{1}{\frac{F(1)}{\Phi^{r}F(1)}}{\cke{\partial_{n}^{k}}}$$
n'est pas triviale si et seulement si l'une des deux conditions équivalentes suivantes est
satisfaite: 
\begin{enumerate}
\item l'élément $\omega^{k+1}\left(u^{2^{r}}\right)$ n'est pas nul;
\item il existe $v$ dans $R^{k}$ tel que $\partial_{r}^{k}(v)=u^{2^{r}}$ et que
$\partial^{k}(v)-u^{2^{r}}$ n'est pas trivial. 
\end{enumerate}
\end{prop}
\begin{proof}
Le module $R^{k+1}$ est $1-$connexe et 
\begin{align*}
\left(R^{k+1}\right)^{1}&\cong\left(\tilde{\h}^{*}\z/2\right)^{1}\cong\langle u\rangle.
\end{align*} 
Donc il résulte de l'exactitude de la résolution 
$\left\{I^{\bullet},\partial^{\bullet}\right\}$ de $F(1)$ que $N^{k+2}$ contient un facteur direct
$J(2^{m_{\alpha}})$ tel que
$$\omega^{k+1}(u)=x_{m_{\alpha}}.$$
Parce que $\ell^{k+1}(F(1))=\h_{r}$ alors $$\omega^{k+1}(u^{2^{r-1}})=
Sq^{2^{r-2}}Sq^{2^{r-3}}\ldots Sq^{1}x_{m_{\alpha}}.$$ 
Il suit que $m_{\alpha}\geq r-1$ et l'extension $\mathscr{M}^{k}$ n'est pas triviale si et
seulement si cette inégalité est stricte. C'est équivalent à dire que $\omega^{k+1}(u^{2^{r}})$
n'est
pas nul. En d'autres termes, $u^{2^{r}}$ n'est pas un cobord. Par ailleurs (voir le corollaire
\ref{calcul1}), la suite 
$$\left(R^{k}\right)^{2^{r}}\to \left(R^{k+1}\right)^{2^{r}}\to
\left(R^{k+2}\right)^{2^{r}}$$
est exacte  donc il existe $v\in \left(R^{k}\right)^{2^{r}}$ tel
que $\partial^{k}_{r}(v)=u^{2^{r}}$. Il en résulte que $\partial^{k}(v)\neq u^{2^{r}}$. Alors,
la trivialité
de l'extension $\mathscr{M}^{k}$ est équivalente à la non-trivialité de
$\partial^{k}(v)-u^{2^{r}}$.
\end{proof}
\begin{defi}
Soit $k=2^{r}(2l+1)-1$. L'extension $\mathscr{M}^{k}$ est dite
$Sq^{1}-$non-triviale si il existe $v\in R^{k}$ tel que
$Sq^{1}\left(\partial^{k}(v)-u^{2^{r}}\right)\neq 0$ et que $\partial_{r}^{k}(v)=u^{2^{r}}$. 
\end{defi}
On constate que la $Sq^{1}-$non-trivialité entraîne la non-trivialité des extensions. Alors le
théorème \ref{exnontriv1} se découle du lemme suivant.
\begin{lemm}\label{prelem1}
Si $\mathscr{M}^{2k+1}$ est $Sq^1-$non-triviale pour tout $k< 2^{n}$ n'étant pas de la forme
$2^{m}-2^{l}-1$ où $m>l$, alors il en est de même pour $\mathscr{M}^{2k+1}$ où $k< 2^{n+1}$ n'étant
pas
de la forme $2^{m}-2^{l}-1$ où $m>l$.
\end{lemm} 
\subsection*{Démonstration du théorème principal}
Le théorème \ref{exnontriv1} est démontré par récurrence. La preuve sera divisée en deux étapes. 
D'après le lemme \ref{rs4}, le théorème \ref{exnontriv1} est vrai pour $k<2$. On peut donc établir
l'argument de récurrence. 

Dans un premier temps, on démontre que si $\mathscr{M}^{2^{r}+t},0\leq t\leq 2^{r-1}-2$ n'est pas
$Sq^{1}-$triviale pour $n\geq r$ alors il en est de même pour $\mathscr{M}^{2^{n+1}-2^{r}+t}.$

Puisque si $t$ est de la forme $2^{r}-2^{k}-1, k\leq r-1,$ l'extension
$\mathscr{M}^{2^{r}+t}$ est triviale alors dans un deuxième temps, on montre que
$\mathscr{M}^{2^{n+1} +2^{n}-2^{r}-1}$ est $Sq^{1}-$non triviale pour $1\leq r<n.$ 

\subsubsection*{Première partie de la preuve}
Soit $q=2^{r-1}(2s+1)-1< 2^{n-2}-1$. On se place dans une situation similaire que celle de la
preuve du théorème \ref{princh42}. Alors il existe un morphisme de complexes 
$$\beta^{\bullet}:I^{2^{n}+\bullet}\to
I^{2^{n}+2^{n-1}+\bullet},$$
tel que pour $j\leq 2^{n-2}$
$$\beta^{2j}(u)=u.$$

On va montrer que la $Sq^{1}-$non-trivialité de $\mathscr{M}^{2^{n}+2^{n-1}+2q+1}$ entraîne la
$Sq^{1}-$non-trivialité de $\mathscr{M}^{2^{n}+2q+1}.$

Supposons que $\mathscr{M}^{2^{n}+2^{n-1}+2q+1}$ n'est pas
$Sq^{1}-$triviale. On va montrer que
$u^{2^{r}}\in I^{2^{n}+2q+2}$ 
n'est pas un cobord. On suppose par l'absurde qu'il l'est. Alors il existe $x\in
R^{2^{n}+2q+1}$ tel
que 
$$\partial^{2^{n}+2q+1}(x)=u^{2^{r}}.$$
Il en résulte que
\begin{align*}
u^{2^{r}}=\beta^{2^{n}+2q+2}\left(u^{2^{r}}\right)=\partial^{2^{n}+2^{n-1}+2q+1}\left(\beta^{2^{n
}+2q+1} (x)\right)
\end{align*}
ce qui contredit la non-trivialité de $\mathscr{M}^{2^{n}+2^{n-1}+2q+1}$. Soit $z\in R^{2^{n}+2q+1}$
tel que 
$$\partial_{r}^{2^{n}+2q+1}(x)=u^{2^{r}}.$$ 
Un tel $z$ existe à cause de l'exactitude de la
suite
$$\left(R^{2^{n}+2q+1}\right)^{2^{r}}\to \left(R^{2^{n}+2q+2}\right)^{2^{r}}\to
\left(R^{2^{n}+2q+3}\right)^{2^{r}}.$$
Notant $m:=2^ {n}+2^{n-1}$, il s'ensuit que:
\begin{align*}
\morp{\partial_{r}^{m+2q+1}}{\beta^{2^{n} +2q+1}\left(x\right)}&=u^{2^{r}},\\
\morp{\beta^{2^{n}+2q+2}}{Sq^{1}\left(\partial^{2^{n}+2q+1}(x)-u^{2^{r}}\right)}&=Sq^1\left(\morp{
\partial^{m+2q+1}} {\beta^{2^{n} +2q+1}\left(x\right)}-u^{2^{r}}\right).
\end{align*}
Il en découle que $\mathscr{M}^{2^{n}+2q+1}$ est $Sq^{1}-$non-triviale. 

Par hypothèse de récurrence
de la $Sq^{1}-$non-trivialité de $\mathscr{M}^{2^{n}+2^{n-1}+2q+1}$ pour $q<2^{n-2}-1$, les
$\mathscr{M}^{2^{n}+2q+1}$ ne sont pas $Sq^{1}-$triviales si $q\leq 2^{n-2}-1$ n'étant pas de la
forme
$2^{n-2}-2^{k}-1,k\leq n-2$. 

\subsubsection*{Suite de la preuve}
Soient $q<n-1$ et $u^{2^{q}}\in R^{2^{n}+2^{n-1}-2^{q}}$. Il découle du corollaire \ref{keyfinal}
qu'il existe une $\left(u^{2^{q}}\to
u^{2^{n-1}-1}\right)-$suite (voir la définition \ref{keybefore})
$$\left\{v_i\in
R^{2^{n}+2^{n-1}-2^{q}-1-i}\left|Sq^{1}v_i\neq 0,0\leq i\leq 2^{n-1}-2^{q}-2 \right.\right\}.$$
Comme 
$$\pa{2^n}{u^{2^{n-1}}}=x_0^{2^{n-1}},$$ 
alors
$$\partial^{2^{n}}(u^{2^{n-1}-1})=Sq^{1}v_{2^{n-1}-2^{q}-2}+x_0^{2^{n-1}-2}x_1.$$
Il découle de
$$\partial_{n}^{2^{n-1}+1}\left(x_0^{2^{n-1}-2}x_1\right)=x_0^{n^{n-1}-1}$$
que
$$\partial^{2^{n}-1}(v_{2^{n-1}-2^{q}-2})=Sq^{1}v_{2^{n-1}-2^{q}-3}+x_0^{2^{n-1}-3}x_1.$$
De manière similaire, on obtient:
$$\partial^{2^{n}+2^{n-1}-2^{q}-1}(v_0)=u^{2^{q}}+x_0^{2^{q}-1}x_1.$$
Il résulte que $\mathscr{M}^{2^{n}+2^{n-1}-2^{q}-1}$ est
$Sq^{1}-$non-triviale.

Le lemme \ref{prelem1} permet de conclure le théorème \ref{exnontriv1} par récurrence.

\section{Les cas particuliers}

Le théorème \ref{princh42} permet de concentrer au calcul de $N^{2^{k}+m}$ pour $m<2^{k-1}$. Dans le cadre de cet article, on peut effectuer les calculs pour $1\leq m\leq 4$.
\subsection{Les $N^{2^{k}+2}$}
Les modules $N^{2^{k}+2}$ sont calculés à l'aide du lemme suivant dont la preuve est laissée aux lecteurs:
\begin{lemm}
La suite 
$$N^{2^{k}+1}\to N^{2^{k}+2}$$
fournit les deux premiers termes de la résolution injective minimale de $\h_{k}$.
\end{lemm}
On rappelle que la somme directe $\bigoplus_{n\geq 0}J(n)$ est isomorphe à l'algèbre:
$$\f2[x_i|i\geq 0, x_{i}\in J(2^{i})^{1}]$$
Un morphisme de modules de Brown-Gitler $J(n)\to J(m)$ est déterminé par une opération de Steenrod $\theta$ de degré $n-m$. Ce morphisme est désigné par $\bullet \theta$. 
La détermination de $N^{2^{k}+2}$ se réalise dans la proposition suivante:
\begin{prop}\label{nk2}
Il y a un isomorphisme:
$$N^{2^{k}+2}\cong \bigoplus_{i=0}^{k-3}J(2^{k-1}-2^{i}).$$
Le morphisme $\partial_{n}^{2^{k}+1}$ se présente sous forme matricielle:
$$\left(\bullet Sq^{1},\bullet Sq^{2},\ldots,\bullet Sq^{2^{k-3}}\right)^{t}.$$
\end{prop}
\begin{proof}
On rappelle que la base de Wall de l'algèbre de Steenrod comporte que les produits de $Sq^{2^{n}}$.
Puisque $J(2^{k-1})$ est cocyclique, alors pour un élément $x\in J(2^{k-1})$ il existe un monôme
$Sq^{2^{n_{1}}}Sq^{2^{n_{2}}}\ldots Sq^{2^{n_{k}}}$ tel que
$$Sq^{2^{n_{1}}}Sq^{2^{n_{2}}}\ldots Sq^{2^{n_{k}}} x=x_{0}^{2^{k-1}}.$$
Par l'instabilité, $n_{j}\geq k-2$ et les seuls monômes contenant $Sq^{2^{k-2}}$ sont
$$Sq^{2^{k-2}}Sq^{2^{k-3}}\ldots Sq^{2^{k-i}}, 2\leq i\leq k.$$
Les éléments de $J(2^{k-1})$ correspondant à ces opérations forment un sous module isomorphe à $\h_{k}$.
Alors le composé
$$\h_{k}\to J(2^{k-1})\xrightarrow{\left(\bullet Sq^{1},\bullet Sq^{2},\ldots,\bullet Sq^{2^{k-3}}\right)^{t}} \bigoplus_{i=0}^{k-3}J(2^{k-1}-2^{i})$$
est trivial. Puisque les $J(m)$ sont co-libres, le quotient $J(2^{k-1})/\h_{k}$ s'injecte dans $\bigoplus_{i=0}^{k-3}J(2^{k-1}-2^{i})$. Le lemme en découle.
\end{proof}
\subsection{Les $N^{2^k+3}, k\geq 2$}
Le module $N^{2^k+3}$ est l'enveloppe injective de $\mathscr{M}^{2^k+1}$ qui s'insère dans la suite exacte courte: 
$$\xymatrixrowsep{0.6pc}\xymatrix{
0\ar[r]&\cke{\partial_{n}^{2^k+1}}\ar[r]&\mathscr{M}^{2^k+1}\ar[r]&\ell^{2^k+2}(F(1))\ar[r]&0.
}$$
\begin{prop}
Soit $k\geq 2$. L'extension $$\mathscr{M}^{2^k+1}\in\ex{\u}{1}{\h_{1}}{\cke{\partial_{n}^{2^k+1}}}$$
 est non-triviale et l'enveloppe injective de $\cke{\partial_{n}^{2^k+1}}$ est celle de $\mathscr{M}^{2^k+1}$.
\end{prop}
\begin{proof}
La non-trivialité de $\mathscr{M}^{2^k+1}$ se résulte du théorème  \ref{exnontriv1}. On note $I_0$ l'enveloppe injective de $\cke{\partial_{n}^{2^k+1}}$ et $I_1$ celle de $\mathscr{M}^{2^k+1}$. On constate que $I_0\subset I_1$. La suite exacte courte $$0\to \cke{\partial_{n}^{2^k+1}}\to\mathscr{M}^{2^k+1}\to \h_{1}\to 0$$
montre que $I_1\subset I_0\oplus \h_{1}$. Puisque l'extension est non-triviale, cette inclusion n'est pas stricte. On en déduit que $I_0\cong I_1$.
\end{proof}
 \begin{coro}\label{cor321}
La suite 
	$$N^{2^n-2^k+1}\xrightarrow{\partial_{n}^{2^n-2^k+1}} N^{2^n-2^k+2}\xrightarrow{\partial_{n}^{2^n-2^k+2}} N^{2^n-2^k+3}$$
fournit les trois premiers termes de la résolution injective minimale de $\h_{k}$.
\end{coro}
\subsection{Les $N^{2^k+4}, k\geq 3$}
On rappelle que $N^{2^k+4}$ est l'enveloppe injective de $\ke{_{2^{k}+2}}$ qui se calcule de la suite exacte courte:
$$0\to J(1)\to \cke{\partial}_{n}^{2^{k}+2}\to\ke{_{2^{k}+2}}\to 0.$$
Puisque $J(1)$ est injectif, $N^{2^k+4}\oplus J(1)$ est l'enveloppe injective de $\cke{\partial}_{n}^{2^{k}+2}$.
Alors:
\begin{prop}
Il existe un morphisme $N^{2^k+3}\to N^{2^k+4}\oplus J(1)$ tel que la suite suivante 
$$0\to \h_{k}\to N^{2^k+1}\xrightarrow{\partial_{n}^{2^k+1}} N^{2^k+2}\xrightarrow{\partial_{n}^{2^k+2}} N^{2^k+3}\to N^{2^k+4}\oplus J(1)$$
fournit les quatre premiers termes de la résolution injective minimale de $\h_{k}$.	
\end{prop}	
\section{Torsion de Frobenius}
Dans cette section, on calcule les premiers termes de la résolution injective minimale de $\h_{k}$
afin de compléter le résultat principal sur la partie nilpotente de la résolution injective minimale
de $F(1)$. Cette information permet d'étudier l'effet de la torsion de Frobenius sur certains
groupes d'extensions de modules instables.
\subsection{Résolution injective minimale de $\h_{k}$}\label{casp}
La proposition \ref{nk2} fournit les deux premiers termes de la résolution injective minimale de $\h_{k}$:
$$0\to \h_{k}\to J\left(2^{k-1}\right)\xrightarrow{\left(\bullet Sq^{1},\bullet Sq^{2},\ldots,\bullet Sq^{2^{k-3}}\right)^{t}} \bigoplus_{i=0}^{k-3}J\left(2^{k-1}-2^{i}\right)$$ 
On désigne par $J^{k}_{2}$ le troisième terme de la résolution injective minimale de $\h_{k}$ et par
$\partial^{k}_{2}$ la deuxième différentielle. Cette sous-section est consacrée à calculer les 
$J^{k}_{2}$ et $\partial^{k}_{2}$. 

Selon \cite{Wal60}, pour $0\leq j\leq i-2$ ou $j=i$, il existe les opérations de Steenrod
$m^{t}_{i,j}$ où $1\leq t\leq i$ telles que
\begin{align*}
Sq^{2^{i}}Sq^{2^{j}}&=\sum_{t=1}^{i}Sq^{2^{i-t}}m^{t}_{i,j}.
\end{align*}
Alors: 
\begin{lemm}
On a:
\begin{align*}
J_{2}^{k}=\left(\bigoplus_{i=1}^{k-2}J\left(2^{k-1}-2^{i}\right)\right)\bigoplus
\left(\bigoplus_{\substack{1\leq i\leq k-2\\0\leq j\leq
i-2}}^{}J\left(2^{k-1}-2^{i}-2^{j}\right)\right)
\end{align*}
et $\partial^{k}_{2}$ restreint au $J(2^{k-1}-2^{i})$ vers $J(2^{k-1}-2^{n}-2^{m})$ est $\bullet m^{n-i}_{n,m}$.
\end{lemm}
\begin{proof}
Soient deux entiers $n,m$ tels que $0\leq m\leq n-2\leq k-5$. On désigne par $x$ l'élément 
$$x_{0}^{2^{k-1}-2^{n}-2^{m}}\in J(2^{k-1}).$$
Alors
\begin{align*}
Sq^{2^{n}}Sq^{2^{m}}x&=Sq^{2^{m}}Sq^{2^{n}}x\\
&=x_{0}^{2^{k-1}}.
\end{align*}
Il suit que 
$$\left(\bullet Sq^{1},\bullet Sq^{2},\ldots,\bullet Sq^{2^{k-3}}\right)^{t}(x)=x_{0}^{2^{k-1}-2^{n+1}-2^{m}}x_{1}^{2^{n}}+x_{0}^{2^{k-1}-2^{n+1}-2^{m}}x_{1}^{2^{n}}.$$
Si $y\in \bigoplus_{i=0}^{k-3}J\left(2^{k-1}-2^{i}\right)$ on désigne par $\bar{y}$ son image dans $$\cke{\left(\bullet Sq^{1},\bullet Sq^{2},\ldots,\bullet Sq^{2^{k-3}}\right)^{t}}.$$ Alors 
$$\overline{x_{0}^{2^{k-1}-2^{n+1}-2^{m}}x_{1}^{2^{n}}}=\overline{x_{0}^{2^{k-1}-2^{n+1}-2^{m}}x_{1}^{2^{n}}}.$$
On montre que $\overline{x_{0}^{2^{k-1}-2^{n+1}-2^{m}}x_{1}^{2^{n}}}$ est non-trivial. Supposons
par l'absurde qu'il est trivial. Alors il existe $z\in J(2^{k-1})$ tel que 
\begin{equation}\label{eqwal}
\left(\bullet Sq^{1},\bullet Sq^{2},\ldots,\bullet Sq^{2^{k-3}}\right)^{t}(z)=x_{0}^{2^{k-1}-2^{n+1}-2^{m}}x_{1}^{2^{n}}
\end{equation}
Il suit que
$$\left(Sq^{2^{n}}Sq^{2^{m}}+Sq^{2^{m}}Sq^{2^{n}}\right)(z)=x_{0}^{2^{k-1}}.$$
D'après \cite{Wal60}
$$Sq^{2^{n}}Sq^{2^{m}}+Sq^{2^{m}}Sq^{2^{n}}\in \mathcal{A}(n-1)$$
ce qui contredit l'égalité (\ref{eqwal}). De même manière,
$$\overline{x_{0}^{2^{k-1}-2^{n+1}-2^{n+2}}x_{1}^{2^{n}}x_{2}^{2^{n}}}, n\leq k-3$$ ne sont pas
triviaux. Puisque la base de Wall de l'algèbre de Steenrod ne comporte que les produits de
$Sq^{2^{l}}$, alors pour tout élément $v$ de 
$$\cke{\left(\bullet Sq^{1},\bullet Sq^{2},\ldots,\bullet Sq^{2^{k-3}}\right)^{t}}$$
il existe une opération de Steenrod $\theta$ telle que $\theta v$ appartient au sous module $M$ engendré par 
$$\left\{\overline{x_{0}^{2^{k-1}-2^{n+1}-2^{m}}x_{1}^{2^{n}}},\overline{x_{0}^{2^{k-1}-2^{n+1}-2^{n+2}}x_{1}^{2^{n}}x_{2}^{2^{n}}},0\leq m\leq n-2\leq k-5\right\}.$$
Alors $J^{k}_{2}$ est l'enveloppe injective de $M$ et on conclut le lemme.
\end{proof}
\subsubsection{Les cas particuliers}
Le lemme suivant détaille les $N^i$ pour $i\leq 24$.
\begin{prop}
	On a 
	\begin{longtable}{|c||*{11}{c|}}
	\hline
	\rowcolor{Gray}\makebox[0.3em]{$k$}&\makebox[3em]{$1$}&\makebox[3em]{$2$}&\makebox[3em]{
$3$}&\makebox[3em]{$4$}\\
	\hline
	$N^{k}$& $0$ &$0$ &$J(1)$ &$0$\\
	\hline
	\rowcolor{Gray}\makebox[0.3em]{$k$}&\makebox[3em]{$5$}&\makebox[3em]{$6$}&\makebox[3em]{
$7$}&\makebox[3em]{$8$}\\
	\hline
	$N^{k}$& $J(2)$ &$0$ &$J(1)$ &$0$\\
	\hline
	\rowcolor{Gray}\makebox[0.3em]{$k$}&\makebox[3em]{$9$}&\makebox[3em]{$10$}&\makebox[3em]{
$11$}&\makebox[3em]{$12$}\\
	\hline
	$N^{k}$& $J(4)$ &$J(3)$ &$J(2)$ &$0$\\
	\hline
	\rowcolor{Gray}\makebox[0.3em]{$k$}&\makebox[3em]{$13$}&\makebox[3em]{$14$}&\makebox[3em]{
$15$}&\makebox[3em]{$16$}\\
	\hline
	$N^{k}$& $J(2)$ &$0$ &$J(1)$ &$0$\\
	\hline
	\rowcolor{Gray}\makebox[0.3em]{$k$}&\makebox[3em]{$17$}&\makebox[3em]{$18$}&\makebox[3em]{$19$}&\makebox[3em]{$20$}\\
	\hline
	$N^{k}$& $J(8)$ &$J(7)\oplus J(6)$ &$J(6)\oplus J(4)$ &$J(5)$\\\hline
	\rowcolor{Gray}$\partial_{n}^{k}$&$\left(\bullet Sq^{1},\bullet Sq^{2}\right)^{t}$ &$\left(\begin{smallmatrix}
	\bullet Sq^{1}& 0\\ \bullet Sq^{2}Sq^{1}& \bullet Sq^{2}
	\end{smallmatrix}\right)$ &$\left(\bullet Sq^{1}, 0\right)$ & $\bullet Sq^{1}$\\\hline
    \makebox[0.3em]{$k$}& \makebox[3em]{$21$}& \makebox[3em]{$22$}& \makebox[3em]{$23$}& \makebox[3em]{$24$}\\
    	\hline
    \rowcolor{Gray}$N^{k}$& $J(4)$& $J(3)$& $J(2)$& $0$\\ \hline
    $\partial_{n}^{k}$&$\bullet Sq^{1}$ &$\bullet Sq^{1}$ &$0$ & $0$\\\hline	 
	\end{longtable}
\end{prop}
\begin{proof}
Les $N^{k}$ et $\partial_{n}^{k}$ pour $k < 20 $ sont calculés grâces à la sous section
\ref{casp} et au lemme \ref{rs4}. Puisque $\mathscr{M}^{18}\cong\cke{\delta}$ où $$\delta:=
	\ell^{18}(F(1))\hookrightarrow \cke{J(7)\oplus J(6)\xrightarrow{\begin{pmatrix}\bullet Sq^{1}& 0\\\bullet \left(Sq^{2}Sq^{1}\right)&\bullet Sq^{2}\end{pmatrix}}J(6)\oplus J(4)}
	$$ donc $\mathscr{M}^{18}\cong\Sigma^5\f2$ alors $N^{20}\cong J(5)$. 
	
	Parce que, d'une part, $\cke{\partial_n^{19}}\cong\Sigma^{4}\f2$ et d'autre part  $\mathscr{M}^{19}$ est une extension non-triviale dans $\ex{\u}{1}{\ell^{20}(F(1))}{\cke{\partial_n^{19}}}$ alors  $\mathscr{M}^{19}\cong\h_{3}$. Il en découle que $N^{21}$ est isomorphe à $J(4)$.
	
	De manière similaire on obtient $N^{22}=J(3),N^{23}=J(2)$ et $N^{24}=0$.
\end{proof} 
Fixons: $$J(n_1,\ldots,n_k):=\dsum{i=1}{k}{J(n_k)}.$$
De manière analogue, on obtient
\begin{prop}
On a
\begin{center}
\fontsize{7pt}{11pt}	\selectfont
\begin{longtable}{|c||*{11}{c|}}
\hline
\rowcolor{Gray}\makebox[0.3em]{$k$}&\makebox[3em]{$33$}&\makebox[3em]{$34$}&\makebox[3em]{$35$}\\
\hline
$N^{k}$& $J(16)$ &$J(15,14,12)$ &$J(14,12,11,8)$ \\\hline
\rowcolor{Gray}$\partial_{n}^{k}$&$\left(\begin{smallmatrix}
	\bullet Sq^{1}\\\bullet Sq^{2}\\\bullet Sq^{4}
	\end{smallmatrix}\right)$ &$\left(\begin{smallmatrix}
	\bullet Sq^{1}& 0 & 0 \\ \bullet Sq^{2}Sq^{1}& \bullet Sq^{2} & 0\\ \bullet Sq^{4} &  \bullet Sq^{3} &  \bullet Sq^{1} \\\bullet Sq^{6,1} &  \bullet Sq^{4,2}+ Sq^{5,1} &  \bullet Sq^{4} \\ 
	\end{smallmatrix}\right)$ & $\left(\begin{smallmatrix}
		\bullet Sq^{1}& 0 & 0 & 0\\ \bullet Sq^{4}& \bullet Sq^{2} & \bullet Sq^{1} & 0\\ \bullet Sq^{6,3} &  \bullet Sq^{4,2,1} &  \bullet Sq^{4,2}&  \bullet Sq^{3} 
		\end{smallmatrix}\right)$\\
\hline
\makebox[0.3em]{$k$}&\multicolumn{2}{c|}{$36$}    &\makebox[3em]{$37$}\\
\hline
\rowcolor{Gray}$N^{k}$& \multicolumn{2}{c|}{$J(13,10,5)$} &$J(12,4,3)$		\\
\hline
$\partial_{n}^{k}$&\multicolumn{2}{c|}{$\left(\begin{smallmatrix}
				\bullet Sq^{1}& 0 & 0\\\bullet Sq^{6,2,1} & \bullet Sq^{5,1}& 0\\\bullet Sq^{6,2,1} & \bullet Sq^{4,2,1}& \bullet Sq^{2}		\end{smallmatrix}\right)$} &$\left(\begin{smallmatrix}
				\bullet Sq^{1}& 0 & 0 \\ \bullet Sq^{6}& 0 & 0\\ 0 &  \bullet Sq^{1} &  0 
				\end{smallmatrix}\right)$\\ \hline
	\rowcolor{Gray}\makebox[0.3em]{$k$}&\makebox[3em]{$38$}&\makebox[3em]{$39$}&\makebox[3em]{$40$}\\
			\hline
		$N^{k}$& $J(11,6,3)$ & $J(10,2)$ & $J(9)$\\ \hline 		
		\rowcolor{Gray}$\partial_{n}^{k}$&  $\left(\begin{smallmatrix}
					\bullet Sq^{1}& 0 & 0\\ 0& 0 & \bullet Sq^{1}
					\end{smallmatrix}\right)$ &$\left(	\bullet Sq^{1},0 \right)$ & $\bullet Sq^{1}$\\\hline	
	\makebox[0.3em]{$k$}&\makebox[3em]{$41$}&\makebox[3em]{$42$}&\makebox[3em]{$43$}\\ \hline	
\rowcolor{Gray}	$N^{k}$&$J(8)$ & $J(7,6)$ &   $J(6,4)$\\\hline
					$\partial_{n}^{k}$&$\left(\bullet Sq^{1},\bullet Sq^{2}\right)^{t}$ &$\left(\begin{smallmatrix}
						\bullet Sq^{1}& 0\\ \bullet Sq^{2}Sq^{1}& \bullet Sq^{2}
						\end{smallmatrix}\right)$ &$\left(\bullet Sq^{1}, 0\right)$ \\\hline
	\rowcolor{Gray}		\makebox[0.3em]{$k$}&\makebox[3em]{$44$}&\makebox[3em]{$45$}&\makebox[3em]{$46$}\\
					\hline
$N^{k}$& $J(5)$ &$J(4)$	 &$J(3)$	\\\hline				
	\rowcolor{Gray}				$\partial_{n}^{k}$&$	\bullet Sq^{1}$ &$	\bullet Sq^{1}$&$	\bullet Sq^{1}$\\\hline	
	$k$& $47$ &$48$	 &$49$	\\\hline							
\rowcolor{Gray}	$N^{k}$& $J(2)$ &$0$	 &$J(8)$	\\\hline
	$\partial_{n}^{k}$& $0$ &$0$ &$0$\\\hline				
	\end{longtable}	
	\end{center}\end{prop}

On peut donc formuler le résultat principal sur la résolution injective minimale de $F(1)$.

\begin{theo}\label{princh43}
Pour $n\geq 6$ on a:
	\begin{center}
		\fontsize{9pt}{11pt}\selectfont
		\begin{longtable}{|c||*{7}{c|}}
			\hline
			\rowcolor{Gray}\makebox[0.7em]{$k$}& \makebox[0.7em]{$2^{n}-32$}& \makebox[0.7em]{$2^{n}-31$}& \makebox[0.7em]{$2^{n}-30$}& \makebox[0.7em]{$2^{n}-29$}&\makebox[2.3em]{$2^{n}-28$}\\\hline
			$N^{k}$ & $0$& $J(16)$& $J(15,14,12)$ &$J(14,12,11,8)$ & $J(13,10,5)$\\\hline\hline
			\rowcolor{Gray} \makebox[0.7em]{$k$}&\makebox[2.3em]{$2^{n}-27$}& \makebox[0.7em]{$2^{n}-26$}& \makebox[0.7em]{$2^{n}-25$}& \makebox[0.7em]{$2^{n}-24$}& \makebox[0.7em]{$2^{n}-23$} \\\hline
			$N^{k}$& $J(12,4,3)$ & $J(11,6,3)$& $J(10,2)$& $J(9)$& $J(8)$\\\hline
			\rowcolor{Gray}$k$ & \makebox[0.7em]{$2^{n}-22$}& \makebox[0.7em]{$2^{n}-21$}& $2^{n}-20$& $2^{n}-19$& $2^{n}-18$\\\hline
			$N^{k}$ & $J(7,6)$ & $J(6,4)$& $J(5)$& $J(4)$& $J(3)$\\\hline\hline
			\rowcolor{Gray}$k$& $2^{n}-17$& $2^{n}-16$& $2^{n}-15$& $2^{n}-14$& $2^{n}-13$\\ \hline
			$N^{k}$& $J(2)$& $0$&$J(8)$ & $J(7,6)$ & $J(6,4)$\\\hline
			\rowcolor{Gray}$k$& $2^{n}-12$& $2^{n}-11$& $2^{n}-10$& $2^{n}-9$& $2^{n}-8$\\ \hline
			$N^{k}$& $J(5)$ & $J(4)$ & $J(3)$ & $J(2)$ & $0$\\\hline
			\rowcolor{Gray}$k$& $2^{n}-7$& $2^{n}-6$& $2^{n}-5$& $2^{n}-4$& $2^{n}-3$\\ \hline
						$N^{k}$& $J(4)$& $J(3)$&$J(2)$ & $0$ & $J(2)$\\\hline
						\rowcolor{Gray}$k$& $2^{n}-2$& $2^{n}-1$& $2^{n}$& \multicolumn{2}{c|}{$2^{n}+1$}\\ \hline
						$N^{k}$& $0$ & $J(1)$ & $0$ & \multicolumn{2}{c|}{$J(2^{n-1})$}\\\hline
	\rowcolor{Gray}$k$& \multicolumn{3}{c|}{$2^{n}+2$}  & \multicolumn{2}{c|}{$2^{n}+3$}   \\ \hline
							$N^{k}$& \multicolumn{3}{c|}{$J(2^{n-1}-1,2^{n-1}-2,\ldots,2^{n-1}-2^{n-3})$} &  \multicolumn{2}{c|}{$J(2^{n-2})\oplus A_{2^{n}+3}$} \\\hline
		\end{longtable}
	\end{center}
	$A_{2^{n}+3}$ désignant $$\left(\bigoplus_{i=1}^{n-3}J\left(2^{n-1}-2^{i}\right)\right)\bigoplus \left(\bigoplus_{\substack{1\leq i\leq n-2\\0\leq j\leq i-2}}^{}J\left(2^{n-1}-2^{i}-2^{j}\right)\right).$$
\end{theo}

La proposition \ref{calculessentiel} permet de conclure:
\begin{theo}\label{princh44}
	Pour $i\leq 49$ ou $i=2^{n}-2^{5}+t$ avec $0\leq t\leq 2^{5}+2$ et $n>5$, il y a des monomorphismes
	$$\ex{\u}{i}{\Phi^{r} F(1)}{\Phi^{r} F(1)}\hookrightarrow\ex{\u}{i}{\Phi^{r+1} F(1)}{\Phi^{r+1} F(1)}$$
	pour tout $r$.
\end{theo}
\begin{notat} Soient $d$ un entier pair et $2^{n_1}+\cdots+2^{n_k}$ son expression $2-$adique. On note: $$\upsilon(d)=1+n_k-k.$$
\end{notat}
Le théorème suivant est un corollaire du théorème \ref{princh44}.
\begin{theo}
	Pour $d\leq 49$ ou $D=2^{n}-2^{5}+t$ avec $0\leq t\leq 2^{5}+2$ et $n>5$, on a:
	$$\ex{\u}{d}{\Phi^{r}F(1)}{\Phi^{r}F(1)}=\left\{\begin{array}{ll}
	0&\text{ si } 2\not|\hspace{1mm}d,\\
	0&\text{ si } 2\hspace{1mm}|\hspace{1mm}d\text{ et }r<\upsilon(d),\\
	\f2&\text{ sinon}.
	\end{array}\right.$$
\end{theo} 
\begin{theo}\label{conjfaible}
	Soit $d$ un entier pair, on a:
	$$\f2\subset \ex{\u}{d}{\Phi^{r}F(1)}{\Phi^{r}F(1)}\text{ si } r\geq \upsilon(d).$$
	De plus si $r\geq \upsilon(d)$ le morphisme 
	$$\ex{\u}{d}{\Phi^{r}F(1)}{\Phi^{r}F(1)}\to \ex{\F}{d}{I}{I}$$
	est non-trivial.
\end{theo}
\begin{proof}
	Le cas de $d=2^n$ a été calculé à l'aide du lemme \ref{calculessentiel}. En effet  $ \ex{\u}{2^{n}}{\Phi^{r}F(1)}{\Phi^{r}F(1)}\cong\f2$ pour tout $r$ tel que $r\geq n$. On raisonne par récurrence sur la longueur $2-$adique de $d$. On suppose que le théorème est vérifié pour $d$ dont $\alpha(d)<k$. On passe au cas de $\alpha(d)=k$. On note $d=2^{n_1}+2^{n_2}+\cdots+2^{n_k}$ où $1\leq n_1<n_2<\ldots<n_k.$ On désigne par $d_1$ la somme $2^{n_2}+\cdots+2^{n_k}$. Pour $r\geq \upsilon(d)$ le diagramme suivant est donc commutatif:
	$$\xymatrixrowsep{3pc}\xymatrixcolsep{4pc}\xymatrix{
		\ex{\u}{d_1}{\Phi^{r}F(1)}{F(1)}\ar[r]^-{\smile \delta_{n_1}}\ar[d]& \ex{\u}{d}{\Phi^{r-1}F(1)}{F(1)}\ar[d]\\
		\ex{\F}{d_1}{I}{I}\ar[r]^{\smile e_{n_1}}&\ex{\F}{d}{I}{I}
	}$$
	$\delta_{n_1}$ désignant le générateur du groupe $ \ex{\u}{2^{n_1}}{\Phi^{r-1}F(1)}{\Phi^{r}F(1)}$. Par hypothèse de récurrence, le composé 
	$$\ex{\u}{d_1}{\Phi^{r}F(1)}{F(1)}\to\ex{\F}{d_1}{I}{I}\xrightarrow{\smile e_{n_1}}\ex{\F}{d}{I}{I}$$
	est non-trivial. On en déduit que $\f2\subset \ex{\u}{d}{\Phi^{r-1}F(1)}{F(1)}$ et que le morphisme $$\ex{\u}{d}{\Phi^{r}F(1)}{F(1)}\to\ex{\F}{d_1}{I}{I}$$ est non-trivial.
\end{proof}

On peut donc conjecturer:
\begin{conj}\label{conj1}
	Soit $d$ un entier, on a:
	$$\ex{\u}{d}{\Phi^{r}F(1)}{\Phi^{r}F(1)}=\left\{\begin{array}{ll}
	0&\text{ si } 2\not|\hspace{1mm}d,\\
	0&\text{ si } 2\hspace{1mm}|\hspace{1mm}d\text{ et }r<\upsilon(d),\\
	\f2&\text{ sinon}.
	\end{array}\right.$$
	De plus, il y a des monomorphismes
	$$\ex{\u}{d}{\Phi^{r} F(1)}{\Phi^{r} F(1)}\hookrightarrow\ex{\u}{d}{\Phi^{r+1} F(1)}{\Phi^{r+1} F(1)}$$
	pour tout $r$.
\end{conj}
\backmatter
\appendix
\section{Sur l'acyclicité des modules instables injectifs réduits}\label{tecnic}
Rappelons que le carré de l'opération de Bockstein $Sq^{1}$ est trivial. Soit $M$ un module
instable, alors $\left(M^{\bullet},Sq^{1}\right)$ est un complexe. Le module $M$ est dit
$Sq^{1}-$acyclique si $\left(M^{\bullet},Sq^{1}\right)$ est acyclique. D'après \cite{LZ86}, les
modules injectifs réduits connexes sont $Sq^{1}-$acycliques.
\begin{defi}\label{keybefore}
Soient $x$ un élément de $\left(R^{m}\right)^{k}$ et $t\geq 1$ un nombre entier. On dit que $x$
admet  une $t-$ième $Sq^{1}-$pré-image $y\in \left(R^{m-t-1}\right)^{k+t}$ si il existe des éléments
$$\left\{x_i\in R^{m-i-1}\left|0\leq i\leq t-1, Sq^1x_i\neq 0\right.\right\}$$ tels que 
\begin{align*}
\partial_{r}^{m-1}\left(x_{0}\right)&=x,\\
\partial_{r}^{m-t-1}(y)&=Sq^{1}x_{t-1},\\
\partial_{r}^{m-t+i}\left(x_{t-i-1}\right)&=Sq^{1}x_{t-i-2}, 0\leq i\leq t-2.
\end{align*}
Une telle suite $\left(x_{0},x_{1},\ldots,x_{t-1}\right)$ est appelée une $(x\to y)-$suite.
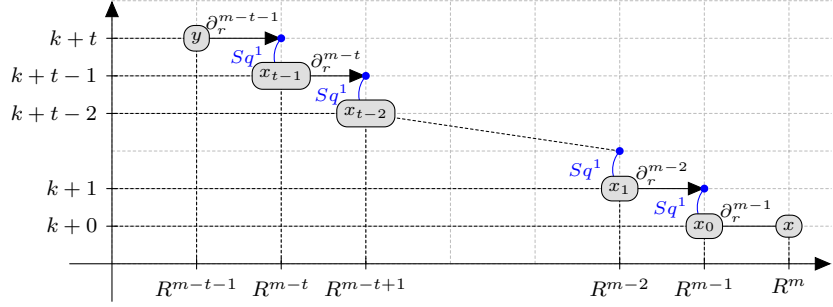
\begin{center}
\definecolor{xdxdff}{rgb}{0.49,0.49,1}
\definecolor{qqqqff}{rgb}{0,0,1}
\definecolor{cqcqcq}{rgb}{0.75,0.75,0.75}
\begin{tikzpicture}[line cap=round,line join=round,>=triangle 45,x=1.125cm,y=0.5cm]
\draw [color=cqcqcq,dash pattern=on 1pt off 1pt, xstep=1.125cm,ystep=0.5cm] (0,0) grid (8.5,7);
\draw[->,color=black] (-0.5,0) -- (8.5,0);
\foreach \x in {1}
\draw[shift={(\x,0)},color=black] (0pt,2pt) -- (0pt,-2pt) node[below] {\footnotesize $R^{m-t-1}$};
\foreach \x in {2}
\draw[shift={(\x,0)},color=black] (0pt,2pt) -- (0pt,-2pt) node[below] {\footnotesize $R^{m-t}$};
\foreach \x in {3}
\draw[shift={(\x,0)},color=black] (0pt,2pt) -- (0pt,-2pt) node[below] {\footnotesize $R^{m-t+1}$};
\foreach \x in {6}
\draw[shift={(\x,0)},color=black] (0pt,2pt) -- (0pt,-2pt) node[below] {\footnotesize $R^{m-2}$};
\foreach \x in {7}
\draw[shift={(\x,0)},color=black] (0pt,2pt) -- (0pt,-2pt) node[below] {\footnotesize $R^{m-1}$};
\foreach \x in {8}
\draw[shift={(\x,0)},color=black] (0pt,2pt) -- (0pt,-2pt) node[below] {\footnotesize $R^{m}$};
\draw[->,color=black] (0,-1) -- (0,7);
\foreach \y in {4}
\draw[shift={(0,\y)},color=black] (2pt,0pt) -- (-2pt,0pt) node[left] {\footnotesize $k+t-2$};
\foreach \y in {5}
\draw[shift={(0,\y)},color=black] (2pt,0pt) -- (-2pt,0pt) node[left] {\footnotesize $k+t-1$};
\foreach \y in {6}
\draw[shift={(0,\y)},color=black] (2pt,0pt) -- (-2pt,0pt) node[left] {\footnotesize $k+t$};
\foreach \y in {0,1}
\draw[shift={(0,1+\y)},color=black] (2pt,0pt) -- (-2pt,0pt) node[left] {\footnotesize $k+\y$};

\clip(-0.45,-0.8) rectangle (8.5,7);

\draw [->] (1,6) -- (2,6);
\draw [->] (2,5) -- (3,5);
\draw [->] (6,2) -- (7,2);
\draw [-] (7,1) -- (8,1);
\draw [dash pattern=on 1pt off 1pt] (3,4)-- (6,3);
\draw [dash pattern=on 1pt off 1pt] (1,6)-- (1,0);
\draw [dash pattern=on 1pt off 1pt] (2,5)-- (2,0);
\draw [dash pattern=on 1pt off 1pt] (3,4)-- (3,0);
\draw [dash pattern=on 1pt off 1pt] (6,2)-- (6,0);
\draw [dash pattern=on 1pt off 1pt] (7,1)-- (7,0);
\draw [dash pattern=on 1pt off 1pt] (8,1)-- (8,0);
\draw [dash pattern=on 1pt off 1pt] (1,6)-- (0,6);
\draw [dash pattern=on 1pt off 1pt] (2,5)-- (0,5);
\draw [dash pattern=on 1pt off 1pt] (3,4)-- (0,4);
\draw [dash pattern=on 1pt off 1pt] (6,2)-- (0,2);
\draw [dash pattern=on 1pt off 1pt] (7,1)-- (0,1);
\begin{scriptsize}
\draw [-,color=qqqqff] (2,5) edge[bend left=40] (2,6);
\draw [-,color=qqqqff] (3,4) edge[bend left=40] (3,5);
\draw [-,color=qqqqff] (6,2) edge[bend left=40] (6,3);
\draw [-,color=qqqqff] (7,1) edge[bend left=40] (7,2);
\draw[color=qqqqff] (1.6,5.5) node {$Sq^1$};
\draw[color=qqqqff] (2.6,4.5) node {$Sq^1$};
\draw[color=qqqqff] (5.6,2.5) node {$Sq^1$};
\draw[color=qqqqff] (6.6,1.5) node {$Sq^1$};
\node[rounded corners,fill=gray!25] (y) at (1,6) [draw] {$y$};
\node[rounded corners,fill=gray!25] (t1) at (2,5) [draw] {$x_{t-1}$};
\node[rounded corners,fill=gray!25] (t3) at (3,4) [draw] {$x_{t-2}$};
\node[rounded corners,fill=gray!25] (t5) at (6,2) [draw] {$x_{1}$};
\node[rounded corners,fill=gray!25] (t7) at (7,1) [draw] {$x_{0}$};
\node[rounded corners,fill=gray!25] (t8) at (8,1) [draw] {$x$};
\draw[color=black] (1.55,6.4) node {$\partial_{r}^{m-t-1}$};
\draw[color=black] (2.65,5.4) node {$\partial_{r}^{m-t}$};
\draw[color=black] (6.5,2.4) node {$\partial_{r}^{m-2}$};
\draw[color=black] (7.5,1.4) node {$\partial_{r}^{m-1}$};
\fill [color=qqqqff] (2,6) circle (1.5pt);
\fill [color=qqqqff] (3,5) circle (1.5pt);
\fill [color=qqqqff] (6,3) circle (1.5pt);
\fill [color=qqqqff] (7,2) circle (1.5pt);
\end{scriptsize}
\end{tikzpicture}
\end{center}
On note:
\begin{align*}
\ro{0}{m}{x} & =\left(\partial_{r}^{m-1}\right)^{-1}(x),\\
\ro{t}{m}{x} & =\left\{y\left|y \text{ est une }t-\text{ième }Sq^{1}-\text{pré-image de
}x\right.\right\},t\geq 1.
\end{align*}
\end{defi}
\begin{lemm}\label{s1}
Soit $n$ un entier et notant $u^{2}\in \tilde{\h}^{*}\z/2\subset R^{2^{n}-2}.$ S'il existe une
$\left(u^{2}\to x\right)-$suite $\left(x_{0},x_{1},\ldots,x_{t-1}\right)$, alors il existe $y$ tel
que $\left(x_{0},x_{1},\ldots,x_{t-1},x\right)$ est une  $\left(u^{2}\to y\right)-$suite.
\end{lemm}
\begin{proof}
Il suffit de démontrer que $Sq^{1}x$ n'est pas trivial. On suppose par l'absurde qu'il est trivial.
Alors il existe $y_{0}\in R^{2^{n}-t-3}$ tel que 
$$Sq^{1}y_{0}=x.$$
Il suit que
\begin{align*}
Sq^{1}\left(\partial_{r}^{2^{n}-t-3}(y_{0})-x_{t-1}\right)&= 0.
\end{align*}
Une simple récurrence montre qu'il existe:
$$\left\{y_{0},y_{1},\ldots,y_{t-1}\left|y_{i}\in R^{2^{n-t-3}+i}\right.\right\}$$
tels que 
\begin{align*}
Sq^{1}\left(\partial_{r}^{2^{n}-t-3+i}(y_{i})-x_{t-1-i}\right)&= 0.
\end{align*}
En particulier 
\begin{align*}
Sq^{1}\left(\partial_{r}^{2^{n}-4}(y_{t-1})-x_{0}\right)&= 0.
\end{align*}
donc $\left(R^{2^{n}-3}\right)^{1}$ n'est pas trivial d'où la contradiction. Le lemme s'en résulte.
\end{proof}
Le lemme \ref{s1} permet de construire une chaîne précise connectant
les éléments $u^{2}\in R^{2^{n}-2}$ et $u^{2^{n}-1}\in R^{0}$.
\begin{lemm}
Il existe une $\left(u^{2}\to u^{2^{n}-1}\right)-$suite 
$$\left\{x_i\in R^{2^{n}-i-3},0\leq i\leq 2^{n}-4, Sq^1x_i\neq 0\right\}$$ telle que pour $2\leq
t\leq n-1,$
\begin{align*}
Sq^{1}x_{2^{t}-3}&=u^{2^{t}}.
\end{align*}
\end{lemm}
\begin{proof}
On résout par récurrence sur $n$. Le cas $n=1$ est trivial. Supposons que le lemme est vrai pour
$n\leq q$, on montre qu'il est vrai pour $n=q+1$. Le produit de Yoneda \cite[proposition 7.2]{FLS94}
$$\ex{\F}{0}{I}{I}\xrightarrow{\smile e^{q-1}}\ex{\F}{2^{q}}{I}{I}$$ 
induit un morphisme de complexes:
$$\gamma_{\bullet}:f(R^{\bullet})\to f(R^{\bullet+2^{q}}).$$
Soit $m$ l'adjoint à droite de $f$ \cite{HLS93}. Il y a un morphisme de
complexes:
$$m(\gamma_{\bullet}):R^{\bullet}\to R^{\bullet+2^{q}}.$$
Par la propriété du produit de Yoneda  \cite{FLS94}, pour tout $l\leq 2^{q}-1,$ 
$$m(\gamma_{2l})(u)=u.$$ 
Par hypothèse de récurrence, il existe une $\left(u^{2}\to
u^{2^{q}-1}\right)-$suite
$$\left\{x_i\in R^{2^{q}-3-i}\left|1\leq i\leq
2^{q}-3,Sq^{1}x_{i}\neq 0\right.\right\}$$ telle que:
\begin{align*}
Sq^{1}x_{2^{t}-3}&=u^{2^{t}-1},2\leq t\leq q.
\end{align*}
Grâce au lemme \ref{s1}, il suffit de démontrer que pour $1\leq i\leq 2^{q}-3,$
$$Sq^{1}m(\gamma_{i})\left(x_{2^{q}-3-i}\right)\neq 0.$$ 
On raisonne par récurrence sur $i$. Parce que $\left(R^{2^{q+1}-3}\right)^{1}$ est trivial, la
$Sq^{1}-$acyclicité de $R^{2^{q+1}-3}$ montre que  $Sq^{1}m(\gamma_{2^{q}-3})(x_0)$ n'est pas
trivial. 
Supposons que pour $1\leq t< i\leq 2^{q}-3,$
$$Sq^{1}m(\gamma_{i})\left(x_{2^{q}-3-i}\right)\neq 0.$$
On passe au cas $i=t$. Si 
$Sq^{1}m(\gamma_{t})\left(x_{2^{q}-3-t}\right)$ est trivial 
alors il existe
$$\left\{z_{j}\in R^{2^{q}+j},t\leq j\leq 2^{q}-4\right\}$$ 
tels que
$$Sq^{1}z_{j+1}=m(\gamma_{j+1})\left(x_{2^{q}-3-j-1}\right)-\partial_{r}^{2^{q}+j}
\left(z_j\right).$$ 
En particulier, on a 
 $$Sq^{1}z_{2^{q}-3}=m(\gamma_{2^{q}-3})\left(x_{0}\right)-\partial_{r}^{2^{q+1}-4}\left(z_{2^{q+1}
-4 }
\right).$$
 Il en résulte que
\begin{align*}
u^{2}&=\partial_{r}^{2^{q+1}-3}\left(m(\gamma_{2^{q}-3})\left(x_{0}\right)\right)\\
&=\partial_{r}^{ 2^ {q+1} -3}\left(\partial_{r}^{2^{q+1} -4 } \left(z_ { 2^ {q}-4}\right)
\right)  \\
&=0
\end{align*}
d'où la contradiction.
\end{proof}
\begin{coro}\label{keyfinal}
Il existe une $\left(u^{2^{t}}\to u^{2^{n}-1}\right)-$suite 
$$\left\{x_i\in R^{2^{n+1}+2^{n}-2^{t}-1-i},0\leq i\leq 2^{n}-2^{t}, Sq^1 x_i\neq 0\right\}$$ telle
que 
\begin{align*}
Sq^{1}x_{2^{k}-2^{t}-1}&=u^{2^{k}},t\leq k\leq n-1.
\end{align*}
\end{coro}
\bibliographystyle{smfalpha}
\bibliography{Thuvienfr}
\end{document}